\documentclass[runningheads,11pt]{llncs}
\usepackage{amssymb}
\usepackage[intlimits]{amsmath}
\usepackage{amsbsy}
\usepackage{graphicx}
\usepackage{subfigmat}
\usepackage{color}
\usepackage{url}
\usepackage{amsfonts}
\usepackage{latexsym} 
\usepackage{wasysym}
\usepackage{newlfont}
\usepackage{enumerate}

\usepackage{stmaryrd}

\usepackage{xspace}
\newcommand{\gismo}{{\fontfamily{phv}\fontshape{sc}\selectfont G\pmb{+}Smo}\xspace}
\newcommand{\tripplenorm}[1]{{\left\vert\kern-0.25ex\left\vert\kern-0.25ex\left\vert #1 
    \right\vert\kern-0.25ex\right\vert\kern-0.25ex\right\vert}}
\numberwithin{equation}{section}
    \urldef{\mailsa}\path|(ulrich.langer, stephen.moore)@ricam.oeaw.ac.at|
    \urldef{\mailsc}\path|neumueller@numa.uni-linz.ac.at|
    \newcommand{\keywords}[1]{\par\addvspace\baselineskip  
     \noindent\keywordname\enspace\ignorespaces#1}
\date{}

\begin{document}
\mainmatter  
\title{SPACE-TIME ISOGEOMETRIC ANALYSIS OF PARABOLIC EVOLUTION EQUATIONS}
\titlerunning{Space-Time IgA of Parabolic Evolution Equations}

\author{Ulrich Langer$^1$, Stephen E. Moore$^1$ and Martin Neum\"uller$^2$}
\authorrunning{U.~Langer, S.~E.~Moore, M.~Neum\"uller}
\institute{ $^1$ 
Johann Radon Institute for Computational and Applied Mathematics,\\ 
		  Austrian Academy of Sciences, 
		  Altenberger Str. 69, 4040 Linz, Austria.\\
		  $^2$ Institute of Computational Mathematics, Johannes Kepler University,\\
Altenberger Str. 69, 4040 Linz, Austria.\\
\mailsa ,   \mailsc 
 }
\maketitle

\begin{abstract}
We present and analyze a new stable space-time Isogeometric Analysis (IgA) method for
the numerical solution of parabolic evolution equations 
in fixed and moving spatial computational domains. 
The discrete bilinear form is elliptic on the IgA space 
with respect to a discrete energy norm. This property together with a corresponding
boundedness property, consistency and approximation results for the IgA spaces
yields an a priori discretization error estimate with respect to the 
discrete norm. The theoretical results are confirmed by several numerical 
experiments with low- and high-order IgA spaces.
\end{abstract}


\section{Introduction}
\label{Sec:Introduction}

Let us consider the parabolic initial-boundary value problem: 
find $u : \overline{Q} \rightarrow \mathbb{R}$ such that
%
\begin{equation}
\label{eqn:ModelProblem}
  \begin{aligned}
  \partial_{t} u(x,t) - \Delta u(x,t) & = f(x,t)\quad \text{for} \, (x,t) \in Q := \Omega \times (0,T),\\
  u(x,t)  & = 0       \quad \, \qquad \text{for} \, (x,t) \in \Sigma := \partial \Omega \times (0,T),   \\
  u(x,0) & = u_{0}(x) \, \quad \, \text{for}\, x \in  \overline{\Omega},
 \end{aligned} 
\end{equation}
%
as the 
typical
model problem for a linear parabolic evolution equation posed 
in the space-time cylinder $ \overline{Q} = \overline{\Omega} \times [0,T]$, 
where $\partial_{t}$ denotes the partial time derivative, $\Delta$ 
is the Laplace operator, $f$ is a given source function, $u_0$ are 
the given initial data, $T$ is the final time, and 
$\Omega \subset \mathbb{R}^{d}$ $(d = 1,2,3)$ denotes the spatial 
computational domain with the boundary $\partial \Omega$.
For the time being, we assume that the domain $\Omega$ is fixed, bounded 
and Lipschitz. In many practical, in particular, industrial applications, 
the domain $\Omega$, 
that is also called physical domain, is usually generated by 
some CAD system, i.e., it can be represented by a single patch or 
multiple patches which are images of the parameter domain 
$(0,1)^d$ by  spline or NURBS maps.

The standard discretization methods in time and space are based on
time-stepping methods combined with some spatial 
discretization technique like the Finite Element Method (FEM)
\cite{LMN:KnabnerAngerman:2003a,LMN:Thomee:2006a,LMN:Lang:2000a}.
The vertical method of lines discretizes first in time and then in space \cite{LMN:Thomee:2006a},
whereas in the horizontal method of lines, also called Rothe's method, 
the discretization starts with respect to (wrt) the time variable \cite{LMN:Lang:2000a}.
The later method has some advantages wrt the development of adaptive techniques.
However, in both approaches, the development of really efficient adaptive techniques 
suffers from the separation of the time and the space discretizations.
Moreover, this separation is even more problematic in parallel computing.
The curse of sequentiality of time affects the construction of efficient parallel 
method and their implementation of massively parallel computers with several 
thousands or even hundreds of  thousands  of cores in a very bad way.

The simplest ideas for space-time solvers are based on
time-parallel integration techniques for ordinary differential equations 
that have a long history, see  
\cite{LMN:Gander:2015a} 
for a comprehensive presentation of this history.
The most popular parallel time integration method
is the parareal method that was introduced by 
Lions, Maday and Turinici in 
\cite{LMN:LionsMadayTurinici:2001a}. 
Time-parallel multigrid methods have also a long history.
In 1984, Hackbusch proposed  the so-called parabolic multigrid 
method that allows a simultaneous execution on a set 
of successive time steps 
\cite{LMN:Hackbusch:1984a}.
Lubich and Ostermann 
\cite{LMN:LubichOstermann:1987a}
introduced
parallel multigrid wavefrom relaxation methods for parabolic problems. 
A comprehensive presentation of these methods and a survey 
of the references until 1993 can be found in the monograph
\cite{LMN:Vandewalle:1993a}.
Vandewalle and Horton investigate the convergence behavior of these
time-parallel multigrid methods by means of Fourier mode analysis
\cite{LMN:HortonVandewalle:1995a}. 
Deshpande et al. provided a rigorous analysis of time domain parallelism
\cite{LMN:DeshpandeMalhotraSchultzDouglas:1995a}.
Very recently, Gander and Neum\"uller have also used the Fourier analysis 
to construct perfectly scaling parallel space-time multigrid methods 
for solving initial value problems for ordinary differential equations
\cite{LMN:GanderNeumueller:2014a} 
and 
initial-boundary value problems for parabolic PDEs 
\cite{LMN:GanderNeumueller:2014b}.
In these two papers, the authors construct stable high-order dG discretizations 
in time slices. In \cite{LMN:Neumueller:2013a} this technique is used to solve the arising linear system of a
space-time dG discretization, which is also stable 
in the case of the decomposition of the space-time cylinder into 
4d simplices (pentatops) for 3d spatial computational domains. This idea
opens great opportunities for flexible discretizations, 
adaptivity and the treatment of changing spatial domains in time
\cite{LMN:NeumuellerSteinbach:2011a,LMN:NeumuellerSteinbach:2013a,LMN:Neumueller:2013a,LMN:KarabelasNeumueller:2015a}.
We also refer to 
\cite{LMN:VegtVen:2002a,LMN:Andreev:2013a,LMN:BankMetti:2013a,LMN:Tezduyar1992:I,LMN:Tezduyar1992:II,LMN:Tezduyar2004,LMN:Tezduyar2006,LMN:Masud1997,LMN:Behr2008,LMN:Lehrenfeld2015}
where different space-time techniques have been developed.
One class of special space-time methods, 
that 
is very relevant to single-patch space-time IgA, 
which we are going to consider in this paper,
is the multiharmonic  or harmonic balanced FEM
that was first used for solving non-linear, time-harmonic eddy current problems 
by electrical engineers, see, e.g., 
\cite{LMN:YamadaBessho:1988}.
Lately the multiharmonic FEM has been applied to 
parabolic and eddy current time-periodic boundary value problems 
as well as to the corresponding optimal control problems, see
\cite{LMN:BachingerLangerSchoeberl:2005a,LMN:KollmannKolmbauerLangerWolfmayrZulehner:2013a,LMN:KolmbauerLanger:2012d,LMN:LangerRepinWolfmayr:2015a}.
Babuska and Janik  already developed $h-p$ versions of the finite element method in space
with $p$ and $h-p$ approximations in time for parabolic 
initial boundary value problems in the papers 
\cite{LMN:BabuskaJanik:1989}
and 
\cite{LMN:BabuskaJanik:1990},
respectively.
In \cite{LMN:SchwabStevenson:2009a}, Schwab and Stevenson 
have recently developed and analyzed space--time adaptive wavelet methods 
for parabolic evolution problems, see also \cite{LMN:CheginiStevenson:2011}. 
Similarly, Mollet proved uniform stability of an abstract 
Petrov-Galerkin discretizations of boundedly invertible operators
and applied this result to space-time discretizations
of linear parabolic problems \cite{LMN:Mollet:2014a}.
Very recently Urban and Patera 
have proved error bounds for reduced basis approximation
to linear parabolic problems \cite{LMN:UrbanPatera:2014a},
whereas Steinbach has investigated conform space-time finite element 
approximations to parabolic problems \cite{LMN:Steinbach:2015}.
We here also mention the papers by Olshanskii and Reusken who have developed
Eulerian space-time finite element method for diffusion problems on evolving surfaces 
\cite{LMN:OlshanskiiReuskenXu:2014a,LMN:OlshanskiiReusken:2014a}.
Our approach uses special time-upwind test functions which are motivated by a space-time streamline diffusion method \cite{LMN:Hansbo:1994,LMN:Johnson:2009,LMN:JohnsonNaevertPitkaeranta:1984,LMN:JohnsonSaranen:1986}
and by a similar approach used in \cite{LMN:BankVassilevski:2014a}.

The increasing interest in highly time-parallel space-time methods 
is certainly connected with the fact that 
parallel computers have rapidly developed with respect to
number of cores, computation speed, memory, availability etc,
but also with the complexity of the problems the people want to solve.
In particular, the optimization of products and processes on the basis 
of computer simulations of the underlying transient processes 
(PDE constrains) foster the development of space-time methods
since the optimality system is basically nothing but  a system of primal and adjoint PDEs 
which are coupled forward and backward in time, see, e.g.,
\cite{LMN:Troeltzsch:2010a}.
To the best of our knowledge, we are not aware of any paper 
on space-time IgA for evolution equations.

In this paper, we present a stable discrete space-time variational formulation for 
the parabolic initial-boundary value problems of the form 
\eqref{eqn:ModelProblem} in the sense that the discrete bilinear 
form $a_h(\cdot,\cdot): V_{0h} \times V_{0h}\rightarrow \mathbb{R}$ 
is elliptic on the IgA space $V_{0h}$ wrt a special discrete norm $\|\cdot\|_h$.
For simplicity, we consider the single-patch case where the space-time cylinder $Q$, 
that is called physical domain,
can be represented by one smooth,
uniformly regular
 spline or NURBS map $\Phi$ of the 
parameter domain ${\widehat Q} = (0,1)^{d+1}$. 
In IgA, that was introduced by Hughes, Cottrell and Bazilevs
in 2005 \cite{LMN:HughesCottrellBazilevs:2005a},
we use the same basis functions for both representing the approximate solution 
and defining the geometrical mapping $\Phi$.
Approximation, stability and error estimates for $h$-refined IgA meshes 
of spatial computational domains can be found in
\cite{LMN:BazilevsBeiraoCottrellHughesSangalli:2006a},
see also the monograph \cite{LMN:CottrellHughesBazilevs:2009a} 
for a comprehensive presentation of the IgA and its mathematical analysis.
Using these approximation results for B-splines and NURBS, 
the $V_{0h}$-ellipticity of the discrete bilinear form $a_h(\cdot,\cdot)$ 
wrt the discrete norm $\|\cdot\|_h$, an appropriate boundedness result
and the consistency, we derive asymptotically optimal discretization 
error estimate in the discrete norm $\|\cdot\|_h$.
Furthermore, we consider moving spatial computational domains 
$\Omega(t) \subset \mathbb{R}^{d}$, $t \in [0,T]$, 
which
always lead to a fix space-time cylinder
$Q = \{ (x,t) \in \mathbb{R}^{d+1}: x \in \Omega(t), t \in (0,T)\} \subset \mathbb{R}^{d+1}$
that can easily be  discretized by IgA.
We again derive a stable space-time IgA scheme and prove 
asymptotically optimal discretization 
error estimate 
in a similar discrete norm $\|\cdot\|_{h,m}$.
Finally, we present a series of numerical experiments for
fixed and moving spatial computational domains 
that support our theoretical results. 
Since the discrete bilinear form is $V_{0h}$-elliptic, 
one may expect that multigrid methods can efficiently 
solve the resulting space-time system of algebraic equation.
Indeed, this is the case as one of our experiments 
with the standard AMG code hypre shows in Subsection~\ref{subsec:ParallelSolution},
where we solve a linear system with 
1.076.890.625
space-time unknowns (dofs) within 156 seconds 
using 16.384 cores. This results demonstrates the great 
potential of our discrete space-time formulation 
for the implementation on massively parallel computers.

The remainder of the paper is organized as follows: 
In Section~\ref{Sec:SpaceTimeVariationalFormulation}, 
we recall the standard space-time variational formulations.
Section~\ref{Sec:StableSpaceTimeIgADiscretizations} is devoted to
the derivation of a stable space-time IgA discretization 
of our parabolic initial-boundary value problem.
In Section~\ref{DiscretizationErrorEstimates}, we derive our 
a priori discretization error estimate.
Section~\ref{Sec:MovingSpatialComputationalDomain} 
deals with the case of moving spatial computational domains.
Our numerical results for fixed and moving spatial domains
are presented and discussed in Section~\ref{NumericalResults}.
Finally, we draw some conclusions in Section~\ref{Conclusions}.


\section{Space-Time Variational Formulations}
\label{Sec:SpaceTimeVariationalFormulation}

Let us first introduce the Sobolev spaces
%
%
$H^{l,k}(Q) =\{ u \in L_2(Q):\, \partial^{\alpha}_x u \in L_2(Q), 
\forall \alpha \,\mbox{with}\, 0 \leq |\alpha| \leq l,\,\partial_t^i u \in L_2(Q), i=0,\ldots,k\}$
of functions defined in the space-time cylinder $Q$, 
where $L_2(Q)$ denotes the space of square-integrable functions, 
$\alpha=(\alpha_1,...,\alpha_d)$ is a multi-index with non-negative integers 
$\alpha_1,...,\alpha_d$, 
$|\alpha| = \alpha_1+\ldots+\alpha_d$, 
$\partial_x^{\alpha} u := \partial^{|\alpha|} u/\partial x^{\alpha} =
\partial^{|\alpha|} u / \partial x_1^{\alpha_1} \ldots \partial x_d^{\alpha_d}$
and 
$\partial_t^i u := \partial^i u/\partial t^i$, see, e.g.,
\cite{LMN:Ladyzhenskaya:1973a,LMN:LadyzhenskayaSolonnikovUralceva:1967a}.
Furthermore, we need the spaces 
$H^{1,0}_{0}(Q) =  \{u \in L_2(Q):\nabla_x u \in [L_2(Q)]^d, u = 0 \, \text{on} \, \Sigma \}$ and 
$H^{1,1}_{0,{\overline 0}}(Q) = \{u \in L_2(Q):\nabla_x u \in [L_2(Q)]^d, \partial_t u \in L_2(Q), u = 0 \, \text{on}\, \Sigma, \, \text{and} \, u = 0\, \text{on} \, \Sigma_T\}$
for introducing the weak space-time formulation of \eqref{eqn:ModelProblem},
where $ \Sigma_{T} := \Omega \times \{T\}$, and 
$\nabla_x u = (\partial u / \partial x_1,\ldots,\partial u / \partial x_d)^\top$
denotes the gradient with respect to the spatial variables.

The standard weak space-time variational formulation of \eqref{eqn:ModelProblem} reads 
as follows: find $u \in H^{1,0}_{0}(Q)$  such that 
\begin{equation}
 \label{eqn:VariationalFormulation}
  a(u,v)  = l(v) \quad \forall v \in  H^{1,1}_{0,{\overline 0}}(Q), 
\end{equation}
with the bilinear form 
\begin{equation}
 \label{eqn:Sec:BilinearForm}
 a(u,v) = - \int_{Q}  u(x,t)\partial_{t} v(x,t) dxdt  + \int_{Q}\nabla_x u(x,t) \cdot \nabla_x v(x,t) dxdt 
\end{equation}
and the linear form
\begin{equation}
 \label{eqn:Sec:LinearForm}
 l(v)  = \int_{Q} f(x,t) v(x,t) dx dt + \int_{\Omega} u_0(x) v(x,0) dx.
\end{equation}
The variational problem \eqref{eqn:VariationalFormulation} is known to have a 
unique weak solution \cite{LMN:Ladyzhenskaya:1973a,LMN:LadyzhenskayaSolonnikovUralceva:1967a}.
In the later books, more general parabolic initial-boundary value problems, 
including other boundary conditions and more general elliptic parts, 
and non-linear  versions
are studied. Moreover, beside existence and uniqueness results, the reader 
also finds useful a priori estimates and regularity results, see also  
\cite{LMN:Troeltzsch:2010a}.
We here mention that time-stepping methods rather based on
line variational formulations which are formulated in 
function spaces of abstract functions mapping the time interval $(0,T)$
into some Sobolev space of functions living on $\Omega$,
see, e.g., 
\cite{LMN:Wloka:1982a,LMN:Zeidler:1990a},
\cite{LMN:Thomee:2006a}, 
and \cite{LMN:Troeltzsch:2010a} 
for the connection of these different formulations.


\section{Stable Space-Time IgA Discretizations}
\label{Sec:StableSpaceTimeIgADiscretizations}
%

In  this section, we briefly recall the definitions of B-splines and NURBS 
basis functions and their use for both the geometrical representation of the  
space-time cylinder $Q$ and the construction of the IgA trial spaces
where we look for 
approximate solutions to our 
variational parabolic evolution problem \eqref{eqn:VariationalFormulation}.
For more details on B-splines and NURBS-based IgA,
we refer to the monograph 
\cite{LMN:CottrellHughesBazilevs:2009a}.
Then we derive our stable space-time IgA scheme that uniquely defines 
an approximate solution $u_h$ in the IgA space. 
This approximate solution can be defined by solving 
one linear system of algebraic equations 
the solution of which is nothing but the vector $\underline{u}_h$ 
of control points for $u_h$.


\subsection{B-splines and NURBS}
\label{Subsec:BSplinesandNURBS}
\begin{figure}[th!]
\begin{center}
  \includegraphics[width=0.45\textwidth]{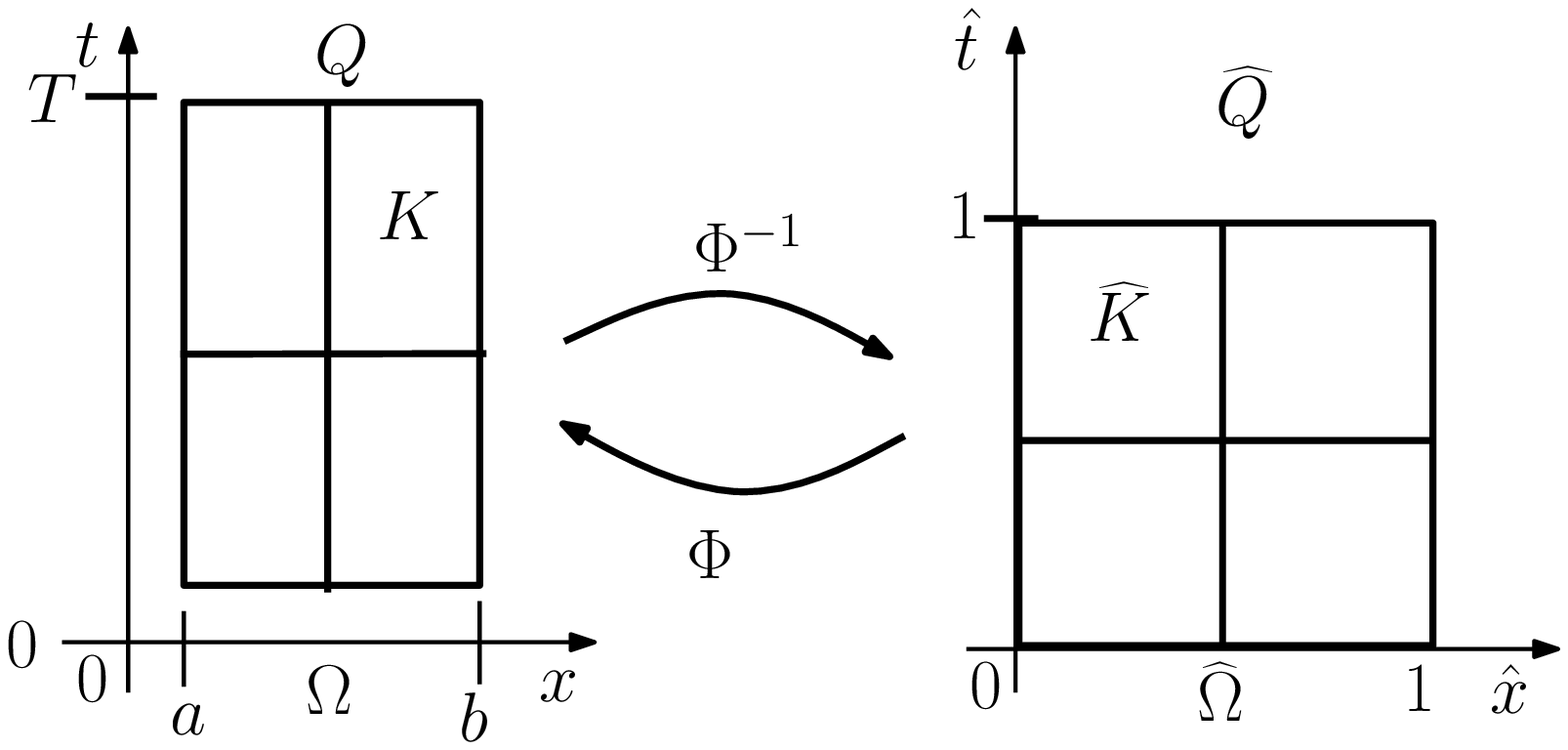}
  \hspace{8.00mm}
  \includegraphics[width=0.46\textwidth]{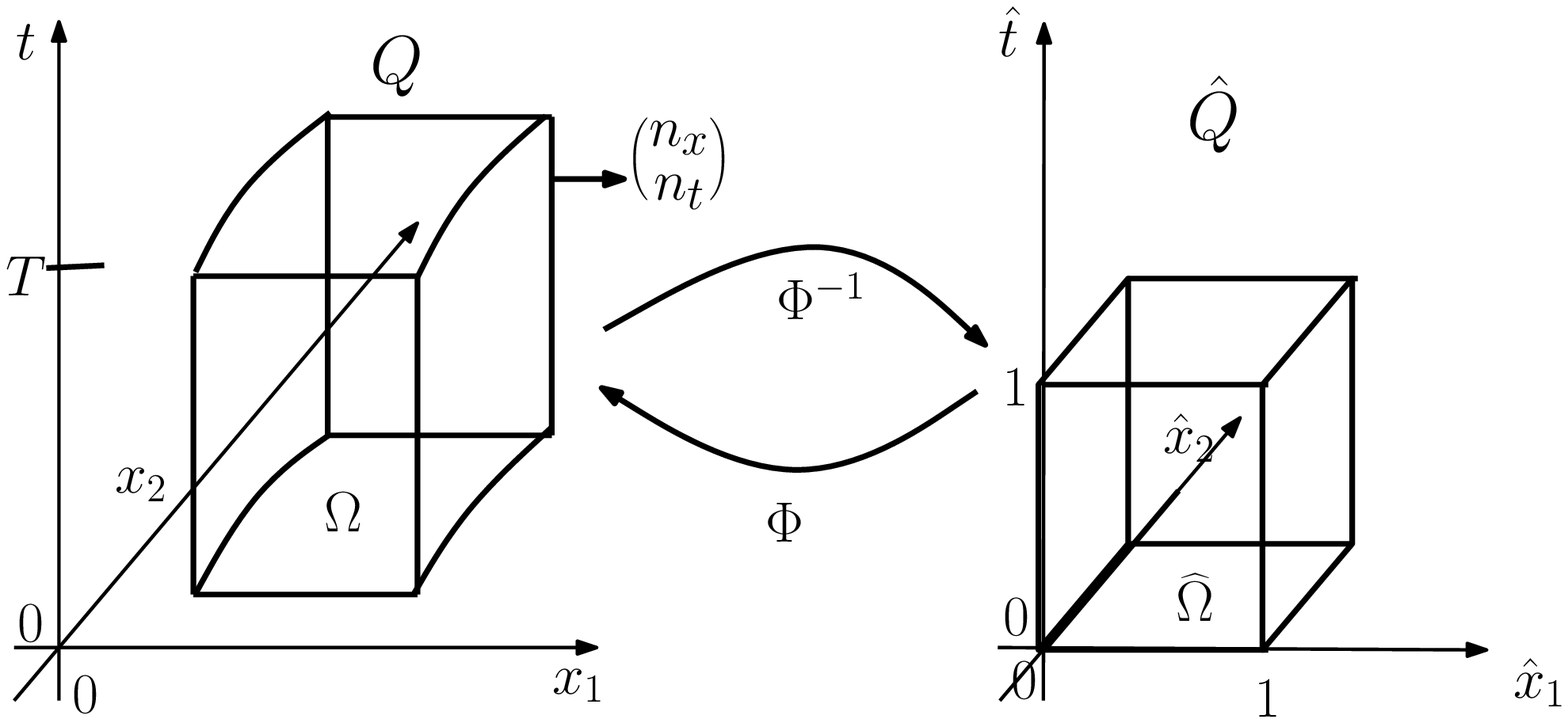}
   \caption{Space-time IgA paraphernalia for $Q \subset \mathbb{R}^{d+1}$ with $d=1$ (left) 
	    and $d=2$ (right).} 
	    \label{Fig:Sec:SpaceTimeMapping}
\end{center}
\end{figure}
Let $p \ge 1$ denote the polynomial degree and $n$ the number of basis functions
defining the B-spline curve.
Then a knot vector is nothing but a set of non-decreasing sequence of real numbers in the parameter 
domain and can be written as  $  \mathrm{\Xi} = \left \{\xi_{1},\ldots,\xi_{n+p+1}\right\}$. 
By convention, we assume that $\xi_1=0$ and $\xi_{n+p+1}=1$.
We also note that the knot value may be repeated indicating its multiplicity $m$.
We always consider  open knot vectors, i.e., knot vectors where the first 
and last knot values appear $p+1$ times, or, in other word, the multiplicity $m$ 
of the first and last knots is $p+1$.
For the one-dimensional parametric domain $(0,1)$,
there is an underlying mesh $\widehat{\mathcal{K}}_h$
consisting of elements $\widehat{K}$ created by the distinct knots. 
We denote the global mesh size by $\hat{h} :=\max\{\hat{h}_K:\widehat{K} \in \widehat{\mathcal{K}}_h \}$,
where $\hat{h}_K:=diam(\widehat{K}) = \mbox{length}(\widehat{K}).$ 
For the time being, we assume that the ratio of the sizes of neighboring 
elements is uniformly bounded from above and below.
In this case, we speak about locally quasi-uniform meshes.

%
%

The univariate B-spline basis functions $\widehat{B}_{i,p} : (0,1) \rightarrow \mathbb{R}$ 
are recursively defined  by means of the Cox-de Boor recursion formula 
as follows:
\begin{align}
 \widehat{B}_{i,0}(\xi)  & = \left\{
   \begin{aligned}
     & 1 & \text{if} \quad & \xi_{i} \leq \xi < \xi_{i+1},\\
     & 0 & \text{else}, &\\
   \end{aligned}
    \right. \\
     \widehat{B}_{i,p}(\xi)  & = \frac{\xi-\xi_{i}}{\xi_{i+p}
     - \xi_{i}}\widehat{B}_{i,p-1}(\xi)
     +\frac{\xi_{i+p+1}-\xi}{\xi_{i+p+1}-\xi_{i+1}}\widehat{B}_{i+1,p-1}(\xi),
\end{align}
where a division by zero is defined to be zero. 
We note that a basis function of degree $p$ is $(p-m)$ times continuously differentiable
across a knot with the multiplicity $m$. If all internal knots have the multiplicity $1$,
then B-splines of degree $p$ are globally $C^{p-1}-$continuous.
In all numerical examples presented in Section~\ref{NumericalResults}, we exactly consider this case.

%

Now we define the multivariate  B-spline 
basis functions on the 
space-time parameter domain 
$\widehat{Q} :=(0,1)^{d+1} \subset \mathbb{R}^{d+1}$ with $d=1,2,$ or $3$ 
as product of the corresponding univariate  B-spline 
basis functions.
To do this, we first define the knot vectors 
$\mathrm{\Xi}_{\alpha} = \left \{\xi^{\alpha}_{1},\ldots,\xi^{\alpha}_{n_{\alpha}+p_{\alpha}+1}\right\}$ 
for every space-time direction $\alpha = 1,\ldots,d+1$.
Furthermore, we introduce the  multi-indicies  ${i}:= (i_1,\ldots,i_{d+1})$ 
and ${p}:= (p_1,\ldots,p_{d+1}),$ the set 
$\overline{\mathcal{I}}=\{{i} = (i_1,\ldots,i_{d+1}) : i_\alpha = 1,2, \ldots, n_\alpha; \;   \alpha = 1,2,  \ldots, d+1\} $.
The product of the univariate  B-spline basis functions 
now gives   multivariate  B-spline basis functions
\begin{equation}
\label{eqn:multivariatebspline}
   \widehat{B}_{{i},{p}}({\xi}) := \prod\limits_{\alpha=1}^{d+1}\widehat{B}_{i_\alpha,p_\alpha}(\xi_{\alpha}),
\end{equation}
where 
$\xi = (\xi_{1},\ldots,\xi_{d+1}) \in \widehat{Q}$.
The univariate and multivariate NURBS basis functions are defined in the parametric domain by means of the corresponding
B-spline basis functions $\{ \widehat{B}_{{i},{p}} \}_{{i} \in \overline{\mathcal{I}}}.$ 
For given $p = (p_1,\ldots,p_{d+1})$ and for all  $i  \in  \overline{\mathcal{I}}$, 
we define the NURBS basis function $\widehat{R}_{{i},{p}}$  as follows:
\begin{align}
\label{eqn:nurbbasisfunction}
   \widehat{R}_{{i},{p}} : (0,1) \rightarrow \mathbb{R},  \quad  \widehat{R}_{{i},{p}}({\xi}) := 
   & \dfrac{\widehat{B}_{{i},{p}}({\xi})w_{{i}}}
   {W({\xi}) } 
\end{align}
with the weighting function
\begin{align}
\label{eqn:nurbsweightfunction}
   W : (0,1) \rightarrow \mathbb{R}, \quad  W({\xi}) := \sum_{\mathbf{i} \in  \overline{\mathcal{I}}} w_{{i}} \widehat{B}_{{i},{p}}({\xi}),
\end{align}
where the weights $w_{{i}}$ are positive real numbers.

The physical  space-time computational domain $Q \subset \mathbb{R}^{d+1}$ 
is now defined from the parametric domain 
$\widehat{Q} = (0,1)^{d+1}$ by the  geometrical mapping 
%
\begin{equation}
\label{eqn:Sec:GeometricalMappingRepresentation}
      \mathrm{\Phi} : \widehat{Q} \rightarrow Q=\mathrm{\Phi}( \widehat{Q}) \subset \mathbb{R}^{d+1}, 
      	\quad  \mathrm{\Phi}(\xi) = \sum_{{ {i} \in\overline{\mathcal{I}}}} \widehat{R}_{{i},{p}}(\xi) \mathbf{P}_{{{i}}},
  \end{equation}
where $\widehat{R}_{{i},{p}}$,  ${ {i} \in \overline{\mathcal{I}}}$, are the multivariate NURBS basis functions 
and  
$\{ \mathbf{P}_{{i}}\}_{{i} \in \overline{\mathcal{I}}}  \subset \mathbb{R}^{d+1}$ are the control points,
see also Figure~\ref{Fig:Sec:SpaceTimeMapping} 
as well as Figures~\ref{fig:Moving2D} and \ref{fig:Moving3D} 
in Subsection~\ref{subsec:MovingSpatialComputationalDomain}.


By means of the geometrical mapping \eqref{eqn:Sec:GeometricalMappingRepresentation},
we define  the physical mesh $\mathcal{K}_h$, 
associated to the computational domain $Q$, whose vertices and elements are the images of the vertices and elements of the 
corresponding underlying mesh 
$\widehat{\mathcal{K}}_h$ 
in the parametric domain $\widehat{Q}$,
i.e., 
$\mathcal{K}_h := \{ K = \mathrm{\Phi}(\widehat{K}) : \widehat{K} \in \widehat{\mathcal{K}}_h \}.$
We denote the global mesh size of the mesh in the  physical domain by 
$h :=\max\{h_K: K \in \mathcal{K}_h\},$ with $h_K := \| \nabla \mathrm{\Phi} \|_{L_{\infty}(K)} \hat{h}_{\hat{K}}$  
and $\hat{h}_{\hat{K}} = \mbox{diam}(\hat{K})$.
Further, we assume that the physical mesh is quasi uniform, i.e., 
there exists a positive constant $C_u$, independent of $h$, such that
\begin{align}
\label{equ:quasiUniformAssumption}
 	 h_K \leq h \leq C_u h_K \qquad \text{for all } K \in \mathcal{K}_h.
\end{align}

Finally, we define the NURBS space  
%
\begin{equation}
\label{eqn:discretenurbsspace}
   V_{h}= \text{span}\{\varphi_{h,{i}} = \widehat{R}_{{i},{p}} \circ \mathrm{\Phi}^{^{-1}}  \}_{{i} \in \overline{\mathcal{I}}}
\end{equation}
 on $Q$ by a push-forward of the NURBS basis functions,
 %
where we assumed that the geometrical mapping $\mathrm{\Phi}$  
is invertible a.e. in  $Q$, with smooth inverse  on each element $K$ 
of the  physical mesh 
$K \in \mathcal{K}_h$, see 
\cite{LMN:BazilevsBeiraoCottrellHughesSangalli:2006a} 
and \cite{LMN:TagliabueDedeQuarteroni:2014a} for more details.
Furthermore, we introduce the subspace 
\begin{equation}
\label{eqn:discretenurbsspacezero}
   V_{0h}
	 := V_{h} \cap H^{1,1}_{0,\underline{0}}(Q)
	  =\{v_h \in V_{h} : v_h|_{\Sigma \cup \Sigma_0} = 0 \}
	 = \text{span}\{\varphi_{h,{i}}\}_{{i} \in \mathcal{I}}
\end{equation}
of all functions from $V_{h}$ fulfilling homogeneous boundary and initial conditions.
%
For simplicity, we below assume the same polynomial degree for all coordinate directions, i.e. $p_\alpha=p$
for all $\alpha=1,2,\ldots,d+1$.


\subsection{Discrete Space-Time Variational Problem}
\label{Subsec:DiscreteVariationalFormulation}
In order to derive our discrete scheme for defining an approximate solution $u_h \in {V}_{0h},$ we multiply 
our parabolic PDE \eqref{eqn:ModelProblem} by a 
time-upwind test function of the form $v_h + \theta h \partial_t v_h$
with an arbitrary $v_h \in {V}_{0h} \subset H^{1,1}_{0,\underline{0}}(Q)$ and a positive constant  $\theta$ which
will be defined later, and integrate over the whole space-time cylinder $Q,$ giving the identity
\begin{align}
 \label{eqn:spacetimeformulation}
 \int_{Q} ( \partial_{t} u_h v_h + \theta h \partial_{t} u_h \partial_{t} v_h  
            &-\Delta u_h(v_h + \theta h \partial_t v_h))\, dx dt \notag  \\ 
            &= \int_{Q} f( v_h + \theta h \partial_{t} v_h) \, dx dt  
\end{align}
that is formally valid for all $v_h \in {V}_{0h} \subset H^{2,1}_{0,\underline{0}}(Q).$
Integration by parts with respect to $x$ in the last term of the bilinear form 
on the left-hand side of \eqref{eqn:spacetimeformulation} gives
 \begin{align*}
 &  \int_{Q}  (\partial_{t} u_h v_h + \theta h \partial_{t} u_h \partial_{t} v_h  
           + \nabla_x u_h \cdot \nabla_x  v_h + \theta h   \nabla_x  u_h \cdot \nabla_x \partial_{t} v_h)\,dxdt \\
 & - \int_{\partial Q} n_x \cdot \nabla_x u_h \,(v_h + \theta h \partial_{t} v_h)\, ds
   = \int_{Q} f v_h  \, dx dt + \theta h  \int_{Q} f \partial_{t} v_h \, dx dt.
\end{align*}
We mention that $\partial_{t} v_h$ is differentiable wrt $x$ due to 
the special tensor product structure of $V_{0h}$.
%
Using the facts that $v_h$ and $\partial_{t} v_h$ are always zero on $\Sigma$,
and the $x$-components $n_x=(n_1,\ldots,n_d)^\top$ of the normal
$n = (n_1,\ldots,n_d,n_{d+1})^\top = (n_x,n_t)^\top$ is zero on $\Sigma_0$ and $\Sigma_T$,
we arrive at our discrete scheme:
find $u_h \in {V}_{0h}$ such that 
\begin{align}
 \label{eqn:DiscreteVariationalFormulation}
  a_h(u_h,v_h)  = l_h(v_h) \quad \forall v_h \in {V}_{0h},
\end{align}
where 
\begin{align}
 a_h(u_h,v_h) = &  \int_{Q}  (\partial_{t} u_h v_h + \theta h \partial_{t} u_h \partial_{t} v_h  \notag \\
          & + \nabla_x u_h \nabla_x  v_h 
	    + \theta h   \nabla_x  u_h \cdot \nabla_x \partial_{t} v_h)\, dxdt,
\label{eqn:bilinearform}  \\
   l_h(v_h) = &  \int_{Q} f  \left[ v_h + \theta h \partial_{t} v_h \right]\, dxdt,
\label{eqn:linearform}
\end{align}
and $\theta$ is a positive constant which will be determined later.

Next we will show that the discrete bilinear form \eqref{eqn:bilinearform} is 
${V}_{0h}-$coercive with respect to the 
discrete norm
\begin{align}
\label{eqn:discretenorm}
 \|  v_h  \|_h = \left[ \| \nabla_x v_h \|_{L_2(Q)}^2 + \theta h \|\partial_t v_h \|_{L_2(Q)}^2 + \dfrac{1}{2}\| v_h \|^2_{L_2(\Sigma_T)} \right]^{1/2}.
\end{align}
\begin{remark}
\label{rem:discretenorm}
We note that \eqref{eqn:discretenorm} is a norm on  $V_{0h}$. 
Indeed, if $\| v_h\|_h = 0$ for some $v_h \in {V}_{0h},$ 
then we have $\nabla_x v_h = 0$ and $\partial_t v_h = 0$ in $Q$ 
and also $v_h = 0 $ on $\Sigma_T.$ This means that the function $v_h$ is constant 
in the space-time cylinder $Q.$
Since $v_h = 0$ on the boundary $\Sigma_T$ and all functions of the  discrete space ${V}_{0h}$ 
are continuous, 
we conclude that $v_h = 0$ in the whole space-time domain $Q.$ 
The other norm axioms are obviously fulfilled.
\end{remark}
\begin{lemma}
 \label{lem:Coercivity}
 The discrete bilinear form $a_h(\cdot,\cdot) : {V}_{0h} \times {V}_{0h} \rightarrow \mathbb{R},$
 defined by \eqref{eqn:bilinearform}, is ${V}_{0h}-$coercive 
wrt
the norm $\|\cdot\|_h,$ i.e.
\begin{align}
 \label{eqn:coercive}
   a_h(v_h,v_h) &\geq \mu_c \|v_h\|_h^2, \quad \forall v_h \in  V_{0h} 
 \end{align}
 with $ \mu_c = 1.$
\end{lemma}
\begin{proof}
Using Gauss' theorem and the fact that $n_t = 0$ on $\Sigma$, we get 
\begin{align}
  a_h(v_h,v_h)  =& \int_{Q}  (\partial_{t} v_h  v_h + \theta h (\partial_{t} v_h )^2
                + \nabla_x v_h \cdot \nabla_x  v_h + \theta h   \nabla_x  v_h \cdot \nabla_x \partial_{t} v_h)\,dxdt \notag \\
                 =& \int_{Q} \left[ \frac{1}{2} \partial_t v_h^2 +  \theta h (\partial_{t} v_h )^2 + |\nabla_x  v_h |^2 + \frac{\theta h}{2} \partial_t |\nabla_x v_h|^2 \right]\,dxdt\notag \\
                 =& \frac{1}{2} \int_{\partial Q} v_h^2 n_t ds  +  \theta h \|\partial_{t} v_h \|^2_{L_2(Q)} + \|\nabla_x v_h \|^2_{L_2(Q)} + \frac{\theta h}{2}\int_{\partial Q} |\nabla_x v_h|^2 n_t ds \notag  \\
                =& \frac{1}{2} \left(\| v_h \|^2_{L_2(\Sigma_T)} - \| v_h \|^2_{L_2(\Sigma_0)} \right) +  \theta h \|\partial_{t} v_h \|^2_{L_2(Q)} + \|\nabla_x v_h \|^2_{L_2(Q)} \notag \\
               & + \frac{\theta h}{2} \left( \|\nabla_x v_h\|^2_{L_2(\Sigma_T)} -  \| \nabla_x v_h \|^2_{L_2(\Sigma_0)} \right) \notag  
\end{align}
for all $v_h \in {V}_{0h}.$
Since $v_h(x,0) = 0$ and, therefore, $\nabla_x v_h(x,0)=0$ for all $x \in \overline{\Omega}$,
i.e., $v_h = 0$ and $\nabla_x v_h = 0$ on $\Sigma_0$, we immediately arrive at the estimate
\begin{align}
a_h(v_h,v_h)  =&   \frac{1}{2} \| v_h \|^2_{L_2(\Sigma_T)} +  \theta h \|\partial_{t} v_h \|^2_{L_2(Q)} + \|\nabla_x v_h \|^2_{L_2(Q)} + \frac{\theta h}{2}  \|\nabla_x v_h\|^2_{L_2(\Sigma_T)} \notag \\
         \geq&  \frac{1}{2} \| v_h \|^2_{L_2(\Sigma_T)} +  \theta h \|\partial_{t} v_h \|^2_{L_2(Q)} + \|\nabla_x v_h \|^2_{L_2(Q)} \notag \\ 
    =& \|v_h\|_h^2 \;\; \forall v \in {V}_{0h},  \notag  
\end{align}
which proves our lemma.
\qed
\end{proof}

Lemma~\ref{lem:Coercivity} immediately yields  that the solution $u_h \in {V}_{0h}$ of 
\eqref{eqn:DiscreteVariationalFormulation} is unique. Indeed, let us assume that there 
is another solution $\widetilde{u}_h \in V_{0h}$ such that
$a_h(\widetilde{u}_h,v_h)=l_h(v_h)$ for all $v_h \in V_{0h}$. Then, taking the difference,
we get $a_h(u_h-\widetilde{u}_h,v_h)= 0$ for all $v_h \in V_{0h}$. 
Choosing $v_h = u_h - \widetilde{u}_h \in V_{0h}$ and applying Lemma~\ref{lem:Coercivity},
we immediately obtain the estimate 
$0 \leq \|u_h - \widetilde{u}_h \|_h^2 \leq a_h(u_h - \widetilde{u}_h,u_h - \widetilde{u}_h) = 0$,
i.e., $ u_h - \widetilde{u}_h=0$. Thus, the solution is unique.
Since the discrete variational problem  \eqref{eqn:DiscreteVariationalFormulation} is posed 
in the finite dimensional space $V_{0h}$, the uniqueness yields existence 
of the solution $u_h \in {V}_{0h}$ of \eqref{eqn:DiscreteVariationalFormulation}.

Since we already choose the standard B-spline resp. NURBS basis of $V_{0h}$ 
in \eqref{eqn:discretenurbsspacezero} as 
$V_{0h} = \text{span}\{ \varphi_{h,{i}} \}_{{i}   \in \mathcal{I}}$,
we look for the solution $u_h \in V_{0h}$ of (\ref{eqn:DiscreteVariationalFormulation})
in the form of 
\begin{align*}
 u_h(x,t) = u_h(x_1,\ldots,x_d,x_{d+1}) = 
\sum_{{i} \in \mathcal{I}} u_{{i}} \varphi_{h,{i}}(x,t)
\end{align*}
where $\underline{u}_h:=[u_{{i}}]_{{i}\in \mathcal{I}} \in \mathbb{R}^{N_h = |\mathcal{I}|}$ 
is the unknown vector of control points 
defined by the solution of the system of algebraic equations
\begin{align}
 \label{eqn:algebraicequation}
   K_h \underline{u}_h = \underline{f}_h,
\end{align}
with the $N_h \times N_h$ 
system
matrix $K_h = [K_\mathbf{ij}=a_h(\varphi_{h,{j}},\varphi_{h,{i}})]_{{i,j} \in \mathcal{I}}$
and the
right-hand side 
$\underline{f}_h = [f_\mathbf{i}=l_h(\varphi_{h,{i}})]_{{i} \in \mathcal{I}} 
\in \mathbb{R}^{N_h}$
that can be generated as in a usual IgA code for elliptic problems.
It is again clear from Lemma~\ref{lem:Coercivity} that the stiffness matrix is regular.

In the next section, we derive an a-priori discretization error estimate wrt the $\|\cdot \|_h$ norm. 


\section{Discretization Error Analysis}
\label{DiscretizationErrorEstimates}
At first, we will present the main ingredients, which are necessary for the derivation of our
a priori discretization error estimate, in form of Lemmas.
%
%
%
\begin{lemma}
\label{lem:ContinuoustraceInequality}
Let $K = \Phi (\widehat{K}) \in \mathcal{K}_h$ and $\widehat{K} \in  \mathcal{\widehat{K}}_h.$ 
Then the scaled trace inequality
\begin{align}
\label{eqn:ContinuoustraceInequality}
  \| v \|_{L_2(\partial K)} \leq C_{tr} h^{-1/2}_K \left( \| v \|_{L_2(K)} + h_K | v |_{H^1(K)}\right)  
\end{align}
holds for all $v \in H^1(K)$, where 
$h_K$ denotes the mesh size of $K$ in the physical domain, and 
$C_{tr}$ is a positive constant that only depends
 on the 
shape regularity of the mapping $\Phi$ and the 
constant $C$ in the trace inequality
\begin{equation}
\label{eqn:ContinuoustraceInequality0}
 \|f\|_{L_2(\partial (0,1)^{d+1})} \le C \|f\|_{H^1((0,1)^{d+1})}
\end{equation}
that is valid for all $f \in H^1((0,1)^{d+1})$.
\end{lemma}
\begin{proof}
The proof follows from \eqref{eqn:ContinuoustraceInequality0} by mapping and scaling 
similar to the proof of \cite[Theorem~3.2]{LMN:EvansHughes:2013a}, 
where \eqref{eqn:ContinuoustraceInequality} is proved for $d=2$ 
in such a way that the structure of the constants $C_{tr}$ is explicitly given.
\qed
\end{proof}
\begin{lemma}
\label{lem:InverseInequalities}
Let $K$ be an arbitrary mesh element from  $\mathcal{K}_h$. 
Then the inverse inequalities 
%
\begin{equation}
\label{eqn:inverseinequality}
  \| \nabla v_h \|_{L_2(K)} \leq C_{inv,1} h^{-1}_K \| v_h \|_{L_2(K)}
\end{equation}
and 
\begin{equation}
 \label{eqn:traceinequality}
  \|  v_h \|_{L_2(\partial K)} \leq C_{inv,0} h^{-1/2}_K \| v_h \|_{L_2(K)}
\end{equation}
hold for all $v_h \in V_{h}$, 
where $C_{inv,1}$ and $C_{inv,0}$ are positive constants,
which are independent of 
$h_K$ and $K\in \mathcal{K}_h$.
\end{lemma}
\begin{proof}
The inverse inequality \eqref{eqn:inverseinequality} is a special case of
the inverse inequalities given in 
\cite[Theorem~4.1]{LMN:BazilevsBeiraoCottrellHughesSangalli:2006a}
and
\cite[Theorem~4.2]{LMN:BazilevsBeiraoCottrellHughesSangalli:2006a}. 
The proof of \eqref{eqn:traceinequality} can be found in
\cite[Theorem~4.1]{LMN:EvansHughes:2013a}.
\qed
\end{proof}

We note that estimate \eqref{eqn:inverseinequality} immediately implies 
the inverse inequality
\begin{equation*}
\label{eqn:inverseinequality0}
 \|\partial_t v_h \| _{L_2(K)}  \leq C_{inv,1} h_K^{-1} \| v_h \|_{L_2(K)}
  \quad \forall v_h \in V_{h},
\end{equation*}
because $\partial_t$ is a part of $\nabla = (\nabla_x, \partial_t)^\top$.
Below we need 
the inverse inequality
\begin{equation}
\label{eqn:inverseinequality1}
 \|\partial_t \partial_{x_i}v_h \| _{L_2(K)}  
    \leq C_{inv,1} h_K^{-1} \| \partial_{x_i}v_h \|_{L_2(K)}
\end{equation}
that is obviously valid for all $v_h \in V_{h}$ and $i=1,\ldots,d$ as well.

%
%

To prove the a priori error estimate, we need to show the uniform boundedness 
of the discrete bilinear form $a_h(\cdot,\cdot)$ on $V_{0h,*} \times V_{0h}$,
where the space $V_{0h,*} = H^{1,0}_0(Q) \cap H^{2,1}(Q)+ V_{0h}$ 
is equipped with the norm
%
\begin{equation}
\label{eqn:discretenorm*}
  \|v\|_{h,*}  = \left[\| v \|^2_{h}  + (\theta h)^{-1}\|v_h\|_{L_2(Q)}^2\right]^{1/2}.
\end{equation}
%
%
\begin{lemma}
\label{lem:boundedness}
  The discrete bilinear form $a_h(\cdot,\cdot)$,
  defined by \eqref{eqn:bilinearform}, 
  is uniformly bounded on $V_{0h,*} \times V_{0h}$, 
  i.e., there exists a positive constant $\mu_b$ that does not depend on $h$
  such that
  \begin{equation}
  \label{eqn:boundedness}
     | a_h(u,v_h) |\leq \mu_b \|u\|_{h,*}\|v_h\|_{h},\; 
      \forall \, u \in V_{0h,*}, \forall \, v_h \in V_{0h}.
  \end{equation}
\end{lemma}
\begin{proof}
Let us estimate the bilinear form \eqref{eqn:bilinearform}
\begin{equation*}
 a_h(u,v_h) =   \int_{Q}  (\partial_{t} u v_h + \theta h \partial_{t} u \partial_{t} v_h  \notag \\
 + \nabla_x u \cdot \nabla_x  v_h + \theta h  \nabla_x  u \cdot \nabla_x \partial_{t} v_h)\, dxdt,
\end{equation*}
term by term.
 
For the first term, 
since  
$V_{0h} \subset H^{1,1}_{0,\underline{0}},$ 
we can proceed with an integration by parts wrt $t$, 
and then estimate the resulting terms by Cauchy-Schwarz inequality as follows:
{\allowdisplaybreaks
\begin{align*}
\int_{Q}  \partial_{t} u_h v_h \,dx dt   
   = &  - \int_Q  u \partial_{t} v_h \,dx dt + \int_{\Sigma_T}  u  v_h \,ds \\ 
   \leq &  \left[ (\theta h)^{-1}  \|u \|^2_{L_2(Q)} \right]^{\frac{1}{2}}  
      \left[ \theta h \|\partial_{t}  v_h \|^2_{L_2(Q)} \right]^{\frac{1}{2}} \\
   & + \left[ \|u \|^2_{L_2(\Sigma_T)} \right]^{\frac{1}{2}}  
      \left[  \|v_h \|^2_{L_2(\Sigma_T)} \right]^{\frac{1}{2}} .
\end{align*}
}

The second and third term of \eqref{eqn:bilinearform} 
can easily be bounded by means of Cauchy's inequality as follows:
\begin{align*}
  \theta h \int_{Q}  \partial_{t} u \partial_{t} v_h   \,dx dt 
  & \leq \left[ \theta h \|\partial_{t} u \|^2_{L_2(Q)} 
  \right]^{\frac{1}{2}} \left[\theta h \|\partial_{t} v_h \|^2_{L_2(Q)} \right]^{\frac{1}{2}}.  
\end{align*}
and
\begin{align*}
  \int_Q \nabla_x u \cdot \nabla_x  v_h  \,dx dt 
  & \leq \left[\|\nabla_{x} u \|^2 _{L_2(Q)}\right]^{\frac{1}{2}} 
    \left[ \|\nabla_{x} v_h \|^2 _{L_2(Q)}\right]^{\frac{1}{2}}.
\end{align*}
The final term in the bilinear form is bounded from above by applying the Cauchy-Schwarz inequality,
inverse inequality \eqref{eqn:inverseinequality1}, 
and inequality \eqref{equ:quasiUniformAssumption}
to obtain
   \begin{align*}
    \theta h  \int_{Q} \nabla_x  u_h & \cdot \nabla_x \partial_{t} v_h\, dxdt 
      \leq \left[ \|\nabla_{x} u \|^2 _{L_2(Q)}\right]^{\frac{1}{2}}\left[(\theta h)^2 \| \partial_t \nabla_{x}v_h \|^2_{L_2(Q)} \right]^{\frac{1}{2}} \\
     & =  \left[ \|\nabla_{x} u \|^2 _{L_2(Q)}\right]^{\frac{1}{2}}\left[(\theta h)^2 \sum_{i=1}^{d} 
     			\sum_{K \in \mathcal{K}_h} \| \partial_t (\partial_{x_i} v_h) \|^2_{L_2(K)} \right]^{\frac{1}{2}}\\
     & \leq  \left[ \|\nabla_{x} u \|^2 _{L_2(Q)}\right]^{\frac{1}{2}}\left[(\theta h)^2 C_{inv,1}^2 
     	\sum_{i=1}^{d} \sum_{K \in \mathcal{K}_h} h_K^{-2} \| \partial_{x_i} v_h \|^2_{L_2(K)} \right]^{\frac{1}{2}} \\     
     & \le  \left[  \|\nabla_{x} u \|^2 _{L_2(Q)}\right]^{\frac{1}{2}}\left[C_u^2C_{inv,1}^2 \theta^2 \|\nabla_{x} v_h \|^2_{L_2(Q)} \right]^{\frac{1}{2}}. 
  \end{align*}
%
Combining the terms from above and using Cauchy's inequality, we get
\begin{align*}
a_h(u_h,v_h) &\leq \bigg[(\theta h)^{-1}  \|u \|^2_{L_2(Q)} +  \|u \|^2_{L_2(\Sigma_T)} \\  
  & + \theta h \|\partial_t u \|^2_{L_2(Q)} + \| \nabla_x u \|^2_{L_2(Q)} + 
      \| \nabla_x u \|^2_{L_2(Q)} \bigg]^{\frac{1}{2}} \\
  & \times \bigg[\theta h \|\partial_t v_h \|^2_{L_2(Q)} +  \|v_h \|^2_{L_2(\Sigma_T)}  + \theta h \|		\partial_t v_h \|^2_{L_2(Q)} \\
  &  + \| \nabla_x v_h \|^2_{L_2(Q)} +  C_u^2C_{inv,1}^2 \theta^2 \| \nabla_x v_h \|^2_{L_2(Q)} \bigg]^{\frac{1}{2}}\\
  & \leq \bigg[ 2\| \nabla_x u \|^2_{L_2(Q)} + \theta h \|\partial_t u \|^2_{L_2(Q)}   
  + (\theta h)^{-1}  \|u \|^2_{L_2(Q)} + \|u \|^2_{L_2(\Sigma_T)} \bigg]^{\frac{1}{2}} \\
  & \times \bigg[ (1+ C_u^2C_{inv,1}^2 \theta^2)\| \nabla_x v_h \|^2_{L_2(Q)} + 2\theta h \|\partial_t v_h \|^2_{L_2(Q)}  +  \|v_h \|^2_{L_2(\Sigma_T)}\bigg]^{\frac{1}{2}} \\ 
  & \leq \mu_b \|u\|_{h,*}\|v_h\|_{h},
\end{align*}
where $\mu_b = (2\max\{(1+C_u^2C_{inv,1}^2\theta^2)/2, 1\})^{1/2}.$ 
\qed
\end{proof}

Next we recall some approximation properties of our B-spline resp. NURBS space
that follow from the approximation results proved in 
\cite[Section~3]{LMN:BazilevsBeiraoCottrellHughesSangalli:2006a} 
and \cite[Section~3]{LMN:TagliabueDedeQuarteroni:2014a}.
Indeed, there exists projective operators $\Pi_h: L_2(Q) \rightarrow V_{h}$
that deliver the corresponding asymptotically optimal approximation results.

\begin{lemma}
\label{lem:InterpolationErrorEstimate}
Let $l$ and $s$ be  integers with $0\le l \le s \le p+1$,
and let $v \in H^s(Q) \cap H^{1,1}_{0,\underline{0}}(Q)$.
Then there exist a projective operator $\Pi_{h}$ from $H^{1,1}_{0,\underline{0}}(Q) \cap H^s(Q)$ 
to $V_{0h}$ and a positive generic constant $C_s$ such that 
\begin{align}
\label{eqn:InterpolationErrorEstimate}
  \sum_{K \in \mathcal{K}_{h}}|v - \Pi_{h}v|^{2}_{H^{l}(K)} 
  & \leq C_s  h^{2(s-l)} \|v \|^2_{H^{s}(Q)}, 
\end{align} 
%
where $h$ again denotes the mesh-size parameter in the physical domain,
$p$ denotes the underlying polynomial degree of the B-spline resp.  NURBS,
and the generic constant $C_s$ only depends on $l,s$ and $p$, the shape regularity of the 
physical space-time domain $Q$ described by the mapping $\Phi$
and, in particular, the gradient $\nabla \Phi$ of the mapping $\Phi$,
but not on $h$ and $v$.
\end{lemma}
\begin{proof}
%
The proof follows the proof of the approximation results presented 
in Subsections 3.3 and 3.4 of \cite{LMN:BazilevsBeiraoCottrellHughesSangalli:2006a},
see also \cite[Proposition~3.1]{LMN:TagliabueDedeQuarteroni:2014a}.
\qed
\end{proof}
If $\Pi_{h}v$ belongs to $V_{0h} \cap H^{l}(Q)$ (the multiplicity of the inner knots 
is not larger than $p+1-l$), then estimate \eqref{eqn:InterpolationErrorEstimate} 
immediately yields the global estimate 
\begin{align}
\label{eqn:InterpolationErrorEstimateQ}
   |v - \Pi_{h}v|_{H^{l}(Q)}  
  & \leq C_s^{0.5}  h^{(s-l)} \|v \|_{H^{s}(Q)}.
\end{align} 
%

%
The basic approximation estimates \eqref{eqn:InterpolationErrorEstimate} 
and \eqref{eqn:InterpolationErrorEstimateQ} 
yield estimates of the approximation error $v - \Pi_{h}v$ 
wrt to the $\| \cdot \|_{L_2(\partial Q)}$ - norm as well as 
wrt to the discrete norms $\| \cdot \|_{h}$ and $\| \cdot \|_{h,*}$
which we later need  in  estimation of the discretization error $u-u_h$.
{\allowdisplaybreaks\  
\begin{lemma}
\label{lem:NormInterpolationEstimate}
{\color{black} 
Let $s$ be a positive integer with $1 \le s \le p+1$,
and let $v \in H^s(Q) \cap H^{1,1}_{0,\underline{0}}(Q)$.
Then there exist a projection $\Pi_{h}$ from $H^{1,1}_{0,\underline{0}}(Q)$ 
to $V_{0h}$ and generic positive constants $C_1$, $C_2$ and $C_3$ such that 
\begin{align}
      \|v - \Pi_{h}v \|_{L_2(\partial Q)}  & \leq C_1 h^{s- 1/2} \|v\|_{H^{s}(Q)} 				\label{eqn:BoundaryInterpolation},\\
      \|v - \Pi_{h}v\|_h  & \leq C_2 h^{s - 1} \|v\|_{H^{s}(Q)}, 
	\label{eqn:EnergyNormInterpolation} \\
      \|v - \Pi_{h}v\|_{h,*}  & \leq C_3 h^{s - 1} \|v\|_{H^{s}(Q)}, 	\label{eqn:EnergyNormInterpolation*}
 \end{align} 
where $h$ is the mesh-size parameter in the physical domain,
$p$ denotes the underlying polynomial degree of the B-spline resp.  NURBS,
and the generic constants $C_1$, $C_2$ and $C_3$ 
only depends on $l,s$ and $p$, the shape regularity of the 
physical space-time domain $Q$ described by the mapping $\Phi$
and, in particular, the gradient $\nabla \Phi$ of the mapping $\Phi$,
but not on $h$ and $v$.
}
\end{lemma}
\begin{proof}
{\color{black}
By applying inequality \eqref{eqn:ContinuoustraceInequality}, 
the quasi uniformity assumption \eqref{equ:quasiUniformAssumption}
and Lemma~\ref{lem:InterpolationErrorEstimate},  
the proof of \eqref{eqn:BoundaryInterpolation} is obtained as follows:
\begin{align*}
   \|v - \Pi_{h}v \|_{L_2(\partial Q)}^2 
    & = \sum_{ \substack{K \in \mathcal{K}_{h} \\ \partial K \cap \partial Q \neq \emptyset } } 
	\|v - \Pi_{h}v\|_{L_{2}(\partial K \cap \partial Q)}^2 \\
    & \leq 2 C_{tr}^2  \sum_{K \in \mathcal{K}_{h} }
               \left(  h_{K}^{-1}\|v - \Pi_{h}v\|_{L_{2}(K)}^2 
		  +  h_{K} |v - \Pi_{h}v |_{H^{1}(K)}^2 \right)\\
    & \leq 2 C_{tr}^2  \sum_{K \in \mathcal{K}_{h} }
               \left( C_u h^{-1}\|v - \Pi_{h}v\|_{L_{2}(K)}^2 
		  +  h |v - \Pi_{h}v |_{H^{1}(K)}^2 \right)\\
    & \leq 2 C_{tr}^2  \left( C_u h^{-1}\|v - \Pi_{h}v\|_{L_{2}(Q)}^2 
		  +  h |v - \Pi_{h}v |_{H^{1}(Q)}^2 \right)\\  
     & \leq 2 C_{tr}^2  \left( C_u h^{-1} C_s h^{2s}
		  +  h C_s h^{2(s-1)} \right)\| v\|^2_{H^{s}(Q)}\\                        
 & \leq 2 C_{tr}^2 C_s (C_u+1) h^{2s-1}\| v\|^2_{H^{s}(Q)}
    = C_1 h^{2s-1} \| v\|^2_{H^{s}(Q)}.
\end{align*}
The definition \eqref{eqn:discretenorm} of the norm $\|\cdot\|_h$,
the approximation error estimate \eqref{eqn:InterpolationErrorEstimate} for $l=1$, 
and the estimate \eqref{eqn:BoundaryInterpolation} just proved yield
\begin{align*}
\|v - \Pi_{h}v\|_h^{2} = & \|\nabla_x(v - \Pi_{h}v)\|^{2}_{L_2(Q)}  
			    + \theta h\|\partial_t(v - \Pi_{h}v)\|^{2}_{L_2(Q)} \\
                         & + \frac{1}{2}\|v - \Pi_{h}v\|^{2}_{L_2(\Sigma_T)}\\
		    \leq & \left( C_s h^{2(s-1)} + 
			    \theta h C_s h^{2(s-1)}+1/2hC_1^2 h h^{2(s-1)} \right) 
			    \| v\|^2_{H^{s}(Q)}\\
                       = & (C_s+C_s \theta h + \frac{C_1}{2} h^2) h^{2(s-1)} \| v\|^2_{H^{s}(Q)} = C_2 h^{2(s-1)}\| v\|^2_{H^{s}(Q)},
\end{align*}
which proves the second estimate of Lemma~\ref{lem:NormInterpolationEstimate}.
Now let us prove the last estimate. Using the definition of 
the norm \eqref{eqn:discretenorm*}, the just proven estimate  \eqref{eqn:EnergyNormInterpolation},
and the approximation error estimate \eqref{eqn:InterpolationErrorEstimate} for $l=0$,
we get
{\allowdisplaybreaks
\begin{align*}
\|v - \Pi_{h}v\|_{h,*}^2  
    & = \|v - \Pi_{h}v\|_h^{2}  + (\theta h)^{-1} \|v - \Pi_{h}v\|_{L_2(Q)}^2 \\
    & \leq \left( C_2 h^{2(s-1)}  + (\theta h)^{-1}C_s h^{2s} \right)\| v\|^2_{H^{s}(Q)} \\
    & = \left(C_2 + \theta^{-1} h C_s\right) h^{2(s-1)} \| v\|^2_{H^{s}(Q)} 
                            = C_3  h^{2(s-1)} \| v\|^2_{H^{s}(Q)},
\end{align*}
}
which completes the proof of the lemma. \qed
}

\end{proof}
}

\begin{lemma}
\label{lem:consistency2}
  If the solution $u \in H^{1,0}_0(Q)$  of the variational problem 
  \eqref{eqn:VariationalFormulation} belongs to $H^{2,1}(Q)$,
  then it satisfies the consistency identity
  \begin{equation}
    \label{eqn:consistency}
    a_h(u,v_h) = l_h(v_h) \quad \forall \; v_h \in V_{0h}.
  \end{equation}
\end{lemma}
%
\begin{proof} 
Since $u \in H^{1,0}_0(Q) \cap H^{2,1}(Q)$, integration by parts in \eqref{eqn:VariationalFormulation}
wrt $t$ and $x$ gives the variational identity
 \begin{align}
  \label{eqn:consistencyproof}
  \int_Q f v \,dx dt &+ \int_{\Sigma_0} u_0 v \,ds =   \int_Q  (\partial_t u - \Delta u )v \,dx dt - \int_{\partial Q}  n_t u v \,ds + \int_{\partial Q} n_x \cdot \nabla_x u v \,ds \notag \\
         & = \int_Q  (\partial_t u - \Delta u )v \,dx dt - \int_{\Sigma_T} u v\, ds + \int_{\Sigma_0} u v\, ds + \int_{\partial Q} n_x \cdot \nabla_x u v \, ds \notag\\
          & = \int_Q  (\partial_t u - \Delta u )v \,dx dt 
	      + \int_{\Sigma_0} u v\, ds 
	      \quad \forall v \in H^{1,1}_{0,{\overline 0}}(Q).
 \end{align}
%
If we take a test function $v \in C^{\infty}_0(Q) \subset H^{1,1}_{0,{\overline 0}}(Q)$, then it follows that
 \begin{align*}
  \int_Q f v \,dx dt =   \int_Q  (\partial_t u - \Delta u )v \,dx dt.
 \end{align*}
Since $C^{\infty}_0(Q)$ is dense in $L_2(Q)$, we have 
 \begin{align*}
  \partial_t u - \Delta u = f   \quad \text{in} \quad  L_2(Q).
 \end{align*}
From \eqref{eqn:consistencyproof}, we now conclude that 
\begin{align*}
   \int_{\Sigma_0} u_0 v \,ds    =  
    \int_{\Sigma_0} u v\,ds
\end{align*}
for all $v \in H^{1,1}_{0,{\overline 0}}(Q)$, i.e., $u = u_0$ on $\Sigma_0$.
Since $u \in H^{1,0}_0(Q)$, the Dirichlet trace of $u$ on $\Sigma$ is zero.
This gives us the strong form of our model problem \eqref{eqn:ModelProblem}.

Now, multiplying 
$\partial_t u - \Delta u = f$
by a test function $v_h + \theta h \partial_t v_h$,
where  $v_h \in V_{0h}$, integrating over $Q$, we get    
 \begin{align*}
 \int_Q  (\partial_t u - \Delta u )(v_h + \theta h \partial_t v_h ) \,dx dt 
=\int_Q f (v_h + \theta h \partial_t v_h ) \,dx dt 
\quad \forall v_h \in V_{0h}.
 \end{align*}
Integrating by parts wrt $x$ gives
\begin{align*}
 &\int_Q  \partial_t u (v_h + \theta h \partial_t v_h ) \,dx dt + \nabla_x u \cdot \nabla_x (v_h + \theta h \partial_t v_h ) \,dx dt \\
 & - \int_{\partial Q} n_x \cdot \nabla_x u (v_h + \theta h \partial_t v_h)\, ds =\int_Q f (v_h + \theta h \partial_t v_h ) \,dx dt. 
\end{align*}
 Using the fact that $n_x = 0$ on both $\Sigma_0$ and $\Sigma_T$, and $v_h \in V_{0h}$, we have 
 \begin{align*}
 \int_Q  \partial_t u (v_h + \theta h \partial_t v_h ) \,dx dt &+ \nabla_x u \cdot \nabla_x (v_h + \theta h \partial_t v_h ) \,dx dt \\
  &=\int_Q f (v_h + \theta h \partial_t v_h ) \,dx dt,
 \end{align*}
 which completes the proof.
 \qed
\end{proof}

Now we are in the position to prove the main result for this section,
namely the a priori discretization error estimate 
in the discrete norm $\|\cdot\|_h$.
\begin{theorem}
\label{thm:EnergyNormErrorEstimate}
  Let $u \in H^{1,0}_0(Q) \cap H^{s}(Q)$ with  $s \geq 2$
  be the exact solution of our model problem \eqref{eqn:VariationalFormulation},
  and let $u_h \in V_{0h}$ be the solution to 
  the IgA scheme
  \eqref{eqn:DiscreteVariationalFormulation}.
  Then the discretization error estimate
  \begin{equation}
  \label{eqn:EnergyNormErrorEstimate}
    \|u - u_h\|_h \leq Ch^{t- 1} \|u\|_{ H^{t}(Q)}, 
  \end{equation}
  holds, where $C$ is a generic positive constant, $t = \min\{s,p+1\}$, 
  and $p$ denotes the underlying polynomial degree of the B-splines or  NURBS.
\end{theorem}
%
\begin{proof}
Subtracting the IgA scheme 
\[a_h(u_h,v_h) = l_h(v_h), \quad \forall v_h \in V_{0h}\]
from the consistence identity 
\[a_h(u,v_h) = l_h(v_h), \quad \forall v_h \in V_{0h},\]
we obtain the so-called  Galerkin orthogonality 
\begin{align}
\label{eqn:GalerkinOrthogonality}
    a_h(u - u_h,v_h) = 0, \quad \forall v_h \in V_{0h},
\end{align}
that is crucial for the discretization error estimate.

Using now the triangle inequality, we can estimate the discretization 
error $u - u_h$ as follows
\begin{align}
\label{eqn:triangleinequality}
   \|u - u_h\|_h \leq \|u - \Pi_h u\|_h + \|\Pi_h u - u_h\|_h.
\end{align}
The first term is nothing but the quasi-interpolation error 
that can easily be estimated by means of Lemma~\ref{lem:NormInterpolationEstimate}.
The estimation of the second term on the right-hand side of \eqref{eqn:triangleinequality}
is more involved. 
Using the fact that $\Pi_h u - u_h \in V_{0h}$,
the $V_{0h}$-ellipticity of the bilinear form $a_h(\cdot,\cdot)$ as was shown in Lemma~\ref{lem:Coercivity},
the Galerkin orthogonality \eqref{eqn:GalerkinOrthogonality},
and the boundedness of the discrete bilinear form Lemma~\ref{lem:boundedness},
we can derive the following estimates
%
%
\begin{align*}
  \mu_c  \|\Pi_h u - u_h\|_h ^2 & \leq a_h(\Pi_h u - u_h,\Pi_h u - u_h)\\
				&  = a_h(\Pi_h u - u,\Pi_h u - u_h) \\
				& \leq \mu_b \|\Pi_h u - u\|_{h,*}\| \Pi_h u - u_h\|_{h}.
\end{align*}
Hence, we have
\begin{align}
\label{eqn:discreteinterpolateenergy}
  \|\Pi_h u - u_h\|_h \leq (\mu_b/\mu_c) \|\Pi_h u - u\|_{h,*}.
\end{align}
Inserting \eqref{eqn:discreteinterpolateenergy} into the triangle inequality \eqref{eqn:triangleinequality}
and using the estimates \eqref{eqn:EnergyNormInterpolation*} and \eqref{eqn:EnergyNormInterpolation}
from Lemma~\ref{lem:NormInterpolationEstimate}, we have
\begin{align*}
   \|u - u_h\|_h &\leq \|u - \Pi_h u\|_h + \|\Pi_h u - u_h\|_h \\
                 & \leq \|u - \Pi_h u\|_{h} +  (\mu_b/\mu_c)\|\Pi_h u - u\|_{h,*} \\
                 & \leq C_s(1 +  \mu_b/\mu_c ) h^{t- 1} \|v\|_{H^{s}(Q)},
\end{align*}
%
which proves the discretization error estimate \ref{eqn:EnergyNormErrorEstimate}
with $C = C_s(1 +  \mu_b/\mu_c )$.
\qed
\end{proof}

\section{Moving Spatial Computational Domain}
\label{Sec:MovingSpatialComputationalDomain}
In this section, we will formulate the space-time scheme for a moving spatial domain 
$Q := \{ (x,t) \in \mathbb{R}^{d+1}; x \in \Omega(t), t \in (0,T)\},$ 
where $\Omega(t) 
\subset
\mathbb{R}^{d}$
for $d=1,2,3.$
\begin{figure}[bth!]
\begin{center}
  \includegraphics[width=0.40\textwidth]{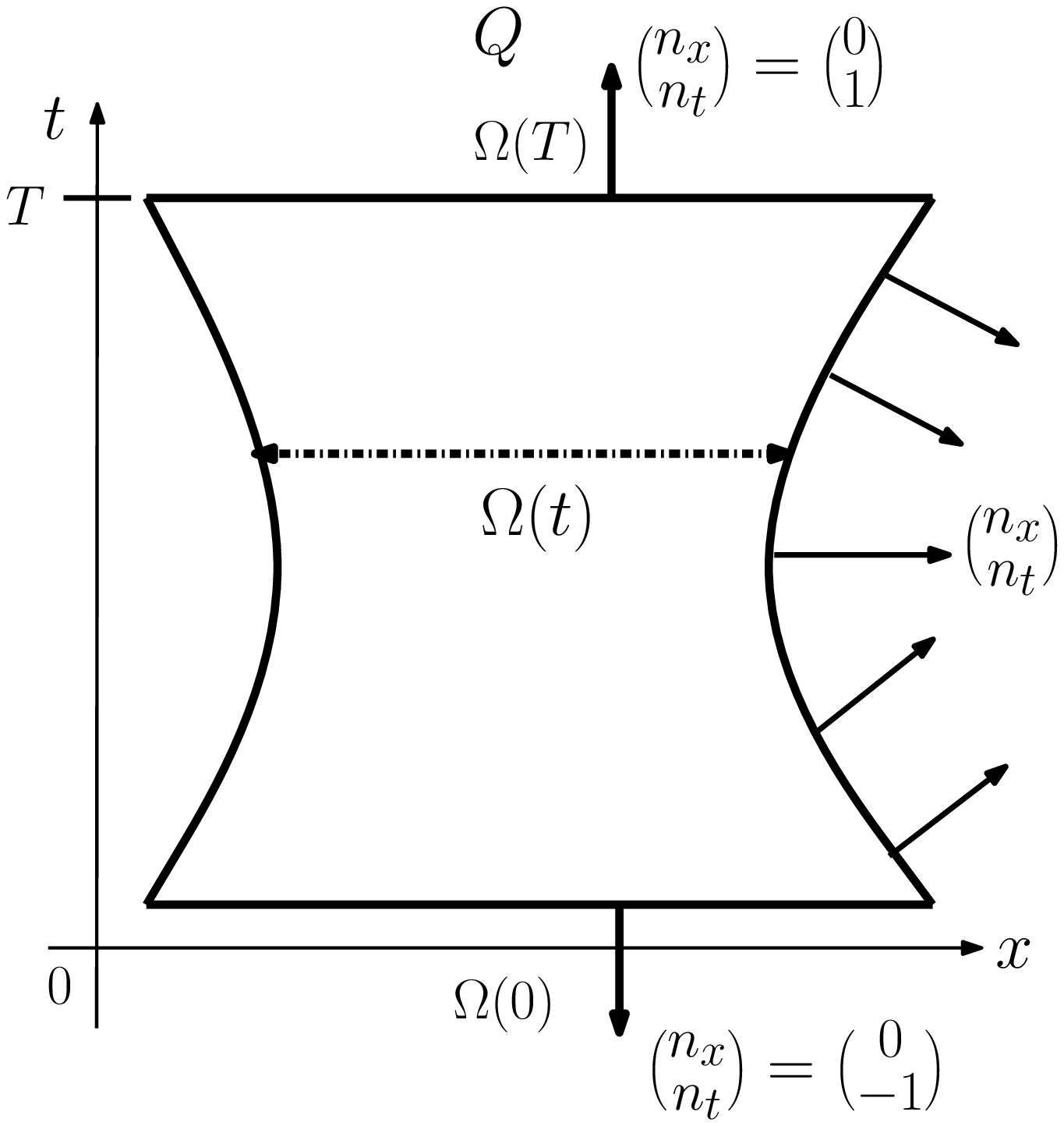}
   \hspace{8.00mm}
  \includegraphics[width=0.50\textwidth]{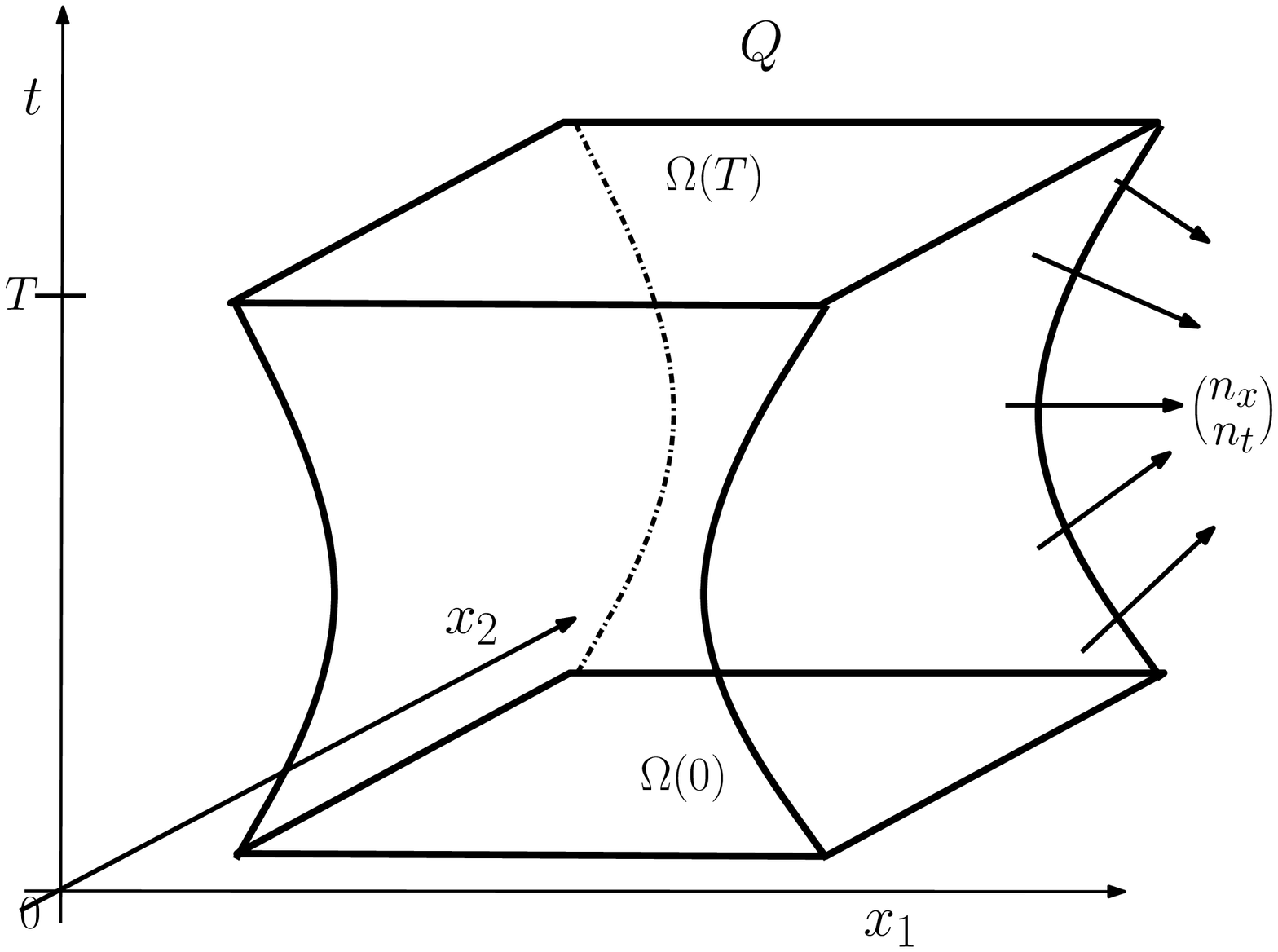}
\end{center}
\caption{Moving spatial domain $\Omega(t)$, $t \in [0,T]$.}
\label{fig:MovingDomainMap}
\end{figure}

Let us now assume that the underlying polynomial degree 
$p \ge 2$ and the multiplicity $m$ of all internal knots is less than 
or equal $p-1$, i.e., ${V}_{0h}$ is always a subset of $C^1(\overline{Q})$.
Thus, the IgA (B-spline or NURBS) space ${V}_{0h}$ is a subspace of
$H^{1,1}_{0,\underline{0}}(Q) \cap H^2(Q).$
Similar to Subsection~\ref{Subsec:DiscreteVariationalFormulation}, we derive the 
IgA
scheme by multiplying 
our parabolic PDE \eqref{eqn:ModelProblem} by a test function of the form $v_h + \theta h \partial_t v_h$
with an arbitrary $v_h \in {V}_{0h}$, 
where 
$\theta$ is a positive constant which will be defined later, 
and integrating over the whole space-time cylinder $Q,$ giving the identity
\begin{align}
 \label{eqn:movingspacetimeformulation}
 \int_{Q} f( v_h + \theta h \partial_{t} v_h) \, dx dt  
 &= \int_{Q} ( \partial_{t} u_h v_h + \theta h \partial_{t} u_h \partial_{t} v_h  \\
            &-\Delta u_h(v_h + \theta h \partial_t v_h))\, dx dt \notag   
\end{align}
that is valid for all $v_h \in {V}_{0h}$.

Integration by parts with respect to $x$ in the last term of the bilinear form 
on the right-hand side of \eqref{eqn:movingspacetimeformulation} gives
 \begin{align*}
   & \int_{Q} f( v_h + 
\theta h \partial_{t} v_h) \, dx dt 
   =  \int_{Q}  (\partial_{t} u_h v_h + \theta h \partial_{t} u_h \partial_{t} v_h  
           + \nabla_x u_h \cdot \nabla_x  v_h \\ 
           & + \theta h   \nabla_x  u_h \cdot \nabla_x \partial_{t} v_h)\,dxdt 
 - \int_{\partial Q} n_x \cdot \nabla_x u \,(v_h + \theta h \partial_{t} v_h)\, ds.
\end{align*}
We now integrate the fourth term on the right-hand side by parts wrt time $t$ and obtain
 \begin{align*}
 & = \int_{Q}  (\partial_{t} u_h v_h + \theta h \partial_{t} u_h \partial_{t} v_h  
           + \nabla_x u_h \cdot \nabla_x  v_h - \theta h  \partial_{t} \nabla_x  u_h \cdot \nabla_x  v_h)\,dxdt \\
 & + \int_{\partial Q} n_t  \nabla_x u_h\cdot \nabla_x v_h\, ds
   - \int_{\partial Q} n_x \cdot \nabla_x u_h \,(v_h + \theta h \partial_{t} v_h)\, ds.
\end{align*}

Using the facts that $v_h \in V_{0h}$ and $n_x$ is zero on $\Sigma_0 \cup \Sigma_T$,
we can continue to write
 \begin{align*}
 &=  \int_{Q}  (\partial_{t} u_h v_h + \theta h \partial_{t} u_h \partial_{t} v_h  
           + \nabla_x u_h \cdot \nabla_x  v_h - \theta h   \partial_{t}  \nabla_x  u_h \cdot \nabla_x v_h)\,dxdt \\
 & + \theta h\int_{\Sigma \cup \Sigma_T} n_t \nabla_x u_h \cdot \nabla_x v_h\, ds
   - \theta h\int_{\Sigma } n_x \cdot \nabla_x u_h \partial_{t} v_h\, ds.
\end{align*}
Considering the terms on the boundary $\Sigma,$ we note the following
 \begin{align*}
\theta h&\int_{\Sigma} [ n_t  \nabla_x u_h \cdot\nabla_x v_h - n_x \cdot \nabla_x u_h \partial_{t} v_h] \, ds \\
& = \theta h\int_{\Sigma} \nabla_x u_h  \cdot [ n_t \nabla_x v_h - n_x \partial_{t} v_h] \, ds = 0,
\end{align*}
since $[n_t \nabla_x v_h - n_x \partial_{t} v_h]$ is the tangential derivative of $v_h$ and $v_h= 0$ on $\Sigma.$ 
Finally, we arrive at our discrete scheme:
find $u_h \in {V}_{0h}$ such that 
\begin{align}
 \label{eqn:MovingDiscreteVariationalFormulation}
  b_h(u_h,v_h)  = l_h(v_h) \quad \forall v_h \in {V}_{0h},
\end{align}
where the bilinear and linear forms are given as follows:
 \begin{align}
 b_h(u_h,v_h)
          & := \int_{Q} (\partial_{t} u_hv_h + \theta h \partial_{t} u_h \partial_{t} v_h + \nabla_x u_h \cdot \nabla_x  v_h \notag \\
          &\quad-\theta h   \partial_{t}  \nabla_x  u_h \cdot \nabla_x v_h) \,dxdt + \theta h \int_{\Sigma_T} \nabla_x u_h \cdot  \nabla_x v_h \,ds, \label{eqn:movingbilinearform}\\ 
   l_h(v_h) & := \int_{Q} f \left[v_h + \theta h \partial_{t} v_h \right]\,dxdt. \notag 
\end{align}

\begin{remark}
    Note that, for $u_h, v_h \in V_{0h}$ with $p \geq 2$, we obtain 
    an alternative representation for the discrete bilinear form
    \begin{align*}
     b_h(u_h,v_h) &= \int_{Q} (\partial_{t} u_hv_h + \theta h \partial_{t} u_h \partial_{t} v_h + \nabla_x u_h \cdot \nabla_x  v_h \notag \\
          &\quad + \theta h \nabla_x  u_h \cdot \nabla_x \partial_{t} v_h) \,dxdt + \theta h \int_{\Sigma} n_t \nabla_x u_h \cdot  \nabla_x v_h \,ds
    \end{align*}
    by integration by parts.
    For non-moving domains, we observe that $n_t = 0$ on $\Sigma$, and we therefore obtain
    \[ b_h(u_h,v_h) = a_h(u_h,v_h) \qquad \text{for all }  u_h, v_h \in V_{0h}. \]
    Hence, in the case of a non-moving domain, the discrete variational formulation \eqref{eqn:DiscreteVariationalFormulation} is equivalent to the discrete problem \eqref{eqn:MovingDiscreteVariationalFormulation}.
\end{remark}

In the next lemma, we show that the discrete bilinear form for the moving spatial domain is $V_{0h}-$coercive wrt to the energy norm
\[ \| v \|_{h,m} := \left[ \| v \|_{h}^2 + \theta h \|\nabla_x v\|^2_{L_2(\Sigma_T)} \right]^{\frac{1}{2}}. \]
\begin{lemma}
 \label{lem:movingcoercivity}
 Let $\theta > 0$ 
 be sufficiently small, in particular, we assume
 \[ \theta < (2 C_{inv,1} C_u)^{-1}, \]
 where $C_{inv,1}$ is the constant from Lemma~\ref{lem:InverseInequalities} and the constant $C_u > 0$ is given by the quasi uniform assumption (\ref{equ:quasiUniformAssumption}).  
 Then the discrete bilinear form $b_h(\cdot,\cdot) : {V}_{0h} \times {V}_{0h} \rightarrow \mathbb{R},$
 defined by \eqref{eqn:movingbilinearform}, is ${V}_{0h}-$coercive 
 wrt the norm $\|\cdot\|_{h,m},$ i.e.
\begin{align}
 \label{eqn:movingcoercive}
   b_h(v_h,v_h) &\geq \mu_c \|v_h\|_{h,m}^2, \quad \forall v \in  V_{0h} 
 \end{align}
 with $ \mu_c = 1/2.$
\end{lemma}
\begin{proof}
Let $v_h \in V_{0h}.$ By means of the definition of the discrete bilinear term $b_h(\cdot,\cdot)$ 
from \eqref{eqn:movingbilinearform} and using Gauss' theorem, we get
\begin{align*}
  b_h(v_h,v_h)  =& \int_{Q}  \left[\partial_{t} v_h  v_h + \theta h (\partial_{t} v_h )^2
                + \nabla_x v_h \cdot \nabla_x  v_h - \theta h  \partial_{t}   \nabla_x  v_h \cdot \nabla_x  v_h \right]\,dxdt \\
                 & + \theta h \int_{\Sigma_T} \nabla_x v_h \cdot  \nabla_x v_h  \,ds\\
                 =& \int_{Q} \left[ \frac{1}{2} \partial_t v_h^2 +  \theta h (\partial_{t} v_h )^2 + |\nabla_x  v_h |^2 - \theta h  \partial_{t}   \nabla_x  v_h \cdot \nabla_x  v_h \right]\,dxdt \\
                 & + \theta h \int_{\Sigma_T} |\nabla_x v_h|^2   \,ds\\
                 =& \frac{1}{2} \int_{\partial Q} v_h^2 n_t ds  +  \theta h \|\partial_{t} v_h \|^2_{L_2(Q)} + \|\nabla_x v_h \|^2_{L_2(Q)} \\& - \theta h \int_{Q} \partial_{t}   \nabla_x  v_h \cdot \nabla_x  v_h dx dt + \theta h \|\nabla_x v_h\|^2_{L_2(\Sigma_T)}
\end{align*}
for all $v \in {V}_{0h}.$ Since $v_h = 0$ on $\Sigma \cup \Sigma_0$, we further obtain
\begin{align*}
b_h(v_h,v_h)  =&   \frac{1}{2} \| v_h \|^2_{L_2(\Sigma_T)} +  \theta h \|\partial_{t} v_h \|^2_{L_2(Q)} + \|\nabla_x v_h \|^2_{L_2(Q)} + \theta h  \|\nabla_x v_h\|^2_{L_2(\Sigma_T)}\\
  &- \theta h \int_{Q} \partial_{t}   \nabla_x  v_h \cdot \nabla_x  v_h dx dt\\
    \geq & \|v_h\|_{h,m}^2 - \theta h \|\partial_t \nabla_x v_h\|_{L_2(Q)}\|\nabla_x v_h\|_{L_2(Q)} \quad \text{for all } v_h \in {V}_{0h}.
\end{align*}
By using the inverse inequality \eqref{eqn:inverseinequality1} and the assumption (\ref{equ:quasiUniformAssumption}) we further obtain similar to the proof of Lemma \ref{lem:boundedness} the estimate
\begin{align*}
 \|\partial_t \nabla_x v_h\|_{L_2(Q)}^2 &= \sum_{K \in \mathcal{K}_{h}} \|\partial_t \nabla_x v_h\|_{L_2(K)}^2 \leq \sum_{K \in \mathcal{K}_{h}} C_{inv,1}^2 h_K^{-2} \|\nabla_x v_h\|_{L_2(K)}^2\\
 &\leq C_{inv,1}^2 C_u^2 h^{-2} \sum_{K \in \mathcal{K}_{h}} \|\nabla_x v_h\|_{L_2(K)}^2 = \left[ C_{inv,1} C_u h^{-1} \| \nabla_x v_h\|_{L_2(Q)} \right]^2.
\end{align*}
Hence, we have for $\theta < (2 C_{inv,1} C_u)^{-1}$ the estimate for the bilinear form
\begin{align*}
 b_h(v_h,v_h) &\geq \|v_h\|_{h,m}^2 - \theta h \|\partial_t \nabla_x v_h\|_{L_2(Q)}\|\nabla_x v_h\|_{L_2(Q)} \\
 &\geq (1 - \theta h C_{inv,1} C_u h^{-1}) \|v_h\|_{h,m}^2  \geq \frac{1}{2} \|v_h\|_{h,m}^2.
\end{align*}
\qed 
\end{proof}
%
To show the boundedness for the bilinear form $b_h(\cdot,\cdot)$, we need the following additional norm
\begin{align*}
 \|v\|_{h,m,\ast} := \left[ \|v\|_{h,m} + (\theta h)^{-1} \|v\|_{L_2(Q)}^2 + (\theta^2 h^2)\|\partial_t \nabla_x v\|_{L_2(Q)}^2 \right]^{\frac{1}{2}}.
\end{align*}

\begin{lemma}
 \label{lem:boundednessmovingdomain}
   The discrete bilinear form $b_h(\cdot,\cdot) : {V}_{0h,\ast} \times {V}_{0h} \rightarrow \mathbb{R},$
 defined by \eqref{eqn:movingbilinearform}, is uniformly bounded on ${V}_{0h,\ast} \times {V}_{0h}$, i.e., there exists a positive constant $\mu_b$ which does not depend on $h$ such that
  \begin{equation}
   \label{eqn:boundednessmovingdomain}
    | b_h(u,v_h) | \leq \mu_b \|u \|_{h,m,*} \|v_h \|_{h,m}
  \end{equation}
  for all $u \in {V}_{0h,\ast}$ and all $v_h \in {V}_{0h}$.
\end{lemma}
\begin{proof}
 We estimate the discrete bilinear form \eqref{eqn:movingbilinearform} 
  \begin{align*}
  b_h(u,v_h) & = \int_{Q} \left(\partial_{t} u v_h + \theta h \partial_{t} u \partial_{t} v_h + \nabla_x u \cdot \nabla_x  v_h -\theta h   \partial_{t}  \nabla_x  u \cdot \nabla_x v_h \right) \,dxdt \\
  &\quad + \theta h \int_{\Sigma_T} \nabla_x u \cdot  \nabla_x v_h \,ds
  \end{align*}
  term by term. For the first, second and third term, we can proceed as in the non-moving case, 
  and we obtain the estimates
  \begin{align*}
   \int_{Q}  \partial_{t} u v_h \,dx dt &\leq  (\theta h)^{-\frac{1}{2}}  \|u \|_{L_2(Q)} (\theta h)^{\frac{1}{2}}  \|\partial_{t}  v_h \|_{L_2(Q)} \\&\quad +  \|u \|_{L_2(\Sigma_T)}  \|v_h \|_{L_2(\Sigma_T)},\\
   \theta h \int_{Q}  \partial_{t} u \partial_{t} v_h   \,dx dt  &\leq (\theta h)^{\frac{1}{2}}  \|\partial_{t}  u \|_{L_2(Q)} (\theta h)^{\frac{1}{2}}  \|\partial_{t}  v_h \|_{L_2(Q)},\\
   \int_{Q}  \nabla_x u \cdot \nabla_x  v_h \,dx dt &\leq \|\nabla_x u \|_{L_2(Q)} \|\nabla_x v_h \|_{L_2(Q)}.
  \end{align*}
  For the two remaining terms, we also use the Cauchy-Schwarz inequality, and we obtain
  \begin{align*}
   \theta h \int_{Q} \partial_{t} \nabla_x u \cdot \nabla_x v_h \,dx dt &\leq \theta h \| \partial_t \nabla_x u \|_{L_2(Q)} \|\nabla_x v_h \|_{L_2(Q)},\\
   \theta h \int_{\Sigma_T} \nabla_x u \cdot  \nabla_x v_h \,ds &\leq (\theta h)^{\frac{1}{2}} \| \nabla_x u \|_{L_2(\Sigma_T)} (\theta h)^{\frac{1}{2}}\|\nabla_x v_h \|_{L_2(\Sigma_T)}.
  \end{align*}
  Combining these estimates, we conclude the statement of this Lemma with
  \begin{align*}
    b_h(u,v_h) &\leq \Big[ (\theta h)^{-1}  \|u \|_{L_2(Q)}^2 + \|u \|_{L_2(\Sigma_T)}^2 + \theta h  \|\partial_{t}  u \|_{L_2(Q)}^2 \\
    &\quad + \|\nabla_x u \|_{L_2(Q)}^2 + \theta^2 h^2 \| \partial_t \nabla_x u \|_{L_2(Q)}^2 + \theta h \| \nabla_x u \|_{L_2(\Sigma_T)}^2 \Big]^{\frac{1}{2}} \\
    &\times \Big[ \theta h  \|\partial_{t}  v_h \|_{L_2(Q)}^2 + \|v_h \|_{L_2(\Sigma_T)}^2 + \theta h  \|\partial_{t}  v_h \|_{L_2(Q)}^2\\
    &\quad + \|\nabla_x v_h \|_{L_2(Q)}^2 + \|\nabla_x v_h \|_{L_2(Q)}^2 + \theta h\|\nabla_x v_h \|_{L_2(\Sigma_T)}^2 \Big]^{\frac{1}{2}} \\
    &\leq 2 \|u \|_{h,m,*} \|v_h \|_{h,m}.
  \end{align*}
\qed
\end{proof}
%
\begin{lemma}
\label{lem:EnergyNormEstimates}
 Let $s$ be a positive integer with $2 \leq s \leq p+1$, and let $v \in H^{s}(Q)$.  
 Then there exist a projection $\Pi_h$  from $H^{1,1}_{0,\underline{0}}(Q) \cap H^s(Q)$ to $V_{0h}$,
 and generic constants $C_1, C_2 > 0$ such that the following error estimates hold
 \begin{align}
      \|v - \Pi_{h}v\|_{h,m}  & \leq C_1 h^{s- 1} \|v\|_{H^{s}(Q)}, \label{eqn:EnergyNormMovingInterpolation} \\
      \|v - \Pi_{h}v\|_{h,m,*}  & \leq C_2 h^{s- 1} \|v\|_{H^{s}(Q)}, \label{eqn:EnergyNormMovingInterpolation*}
 \end{align} 
 where 
 $p$ denotes the underlaying polynomial degree of the B-spline resp. NURBS.
\end{lemma}
\begin{proof}
  By using Lemma \ref{lem:NormInterpolationEstimate} it remains to estimate the additional terms
  \begin{align*}
    h^2 \|\partial_t \nabla_x (v - \Pi_{h}v)\|_{L_2(Q)}^2 \quad \text{and} \quad h \| \nabla_x (v - \Pi_{h}v)\|_{L_2(\Sigma_T)}^2.
  \end{align*}
  With Lemma \ref{lem:InterpolationErrorEstimate} we obtain for $i,j=1,\ldots,d+1$
  \begin{align}\label{equ:errorEstimateMixedDerivates}
   \|\partial_{x_i} \partial_{x_j} (v - \Pi_{h}v)\|_{L_2(Q)} \leq C_s h^{s-2} \|v\|_{H^{s}(Q)}.
  \end{align}
  By using the error estimate \eqref{equ:errorEstimateMixedDerivates} we obtain
  \begin{align*}
   h^2 \|\partial_t \nabla_x (v - \Pi_{h}v)\|_{L_2(Q)}^2 &= h^2 \sum_{i=1}^d \|\partial_t \partial_{x_i} (v - \Pi_{h}v)\|_{L_2(Q)}^2\\
   &\leq h^2 \sum_{i=1}^d C_s^2 h^{2(s-2)} \|v\|_{H^{s}(Q)}^2 \\
   &= \left[\sqrt{d} C_s h^{s-1} \|v\|_{H^{s}(Q)} \right]^2.
  \end{align*}
  With Lemma \ref{lem:ContinuoustraceInequality} and using also \eqref{equ:errorEstimateMixedDerivates} we further obtain
  \begin{align*}
    &h  \| \nabla_x (v - \Pi_{h}v)\|_{L_2(\Sigma_T)}^2 = h \sum_{ \substack{K \in \mathcal{K}_{h} \\ \partial K \cap \Sigma_T \neq \emptyset } } \sum_{i=1}^d \| \partial_{x_i} (v - \Pi_{h}v)\|_{L_2(\partial K \cap \Sigma_T)}^2\\
    & \leq h \sum_{ \substack{K \in \mathcal{K}_{h} \\ \partial K \cap \Sigma_T \neq \emptyset } } \sum_{i=1}^d C_{tr}^2 h_K^{-1} \left( \|\partial_{x_i} (v - \Pi_{h}v)\|_{L_2(K)} + h_K | \partial_{x_i} (v - \Pi_{h}v) |_{H^1(K)} \right)^2\\
    &\leq 2 C_u C_{tr}^2 \sum_{ \substack{K \in \mathcal{K}_{h} \\ \partial K \cap \Sigma_T \neq \emptyset } } \sum_{i=1}^d \left( \|\partial_{x_i} (v - \Pi_{h}v)\|_{L_2(K)}^2 + h^2 | \partial_{x_i} (v - \Pi_{h}v) |_{H^1(K)}^2 \right)\\
    &\leq 2 C_u C_{tr}^2 \left[ \|\nabla_x (v - \Pi_{h}v)\|_{L_2(Q)}^2 + h^2 \sum_{i=1}^d | \partial_{x_i} (v - \Pi_{h}v) |_{H^1(Q)}^2 \right] \\
    &= 2 C_u C_{tr}^2 \left[ \|\nabla_x (v - \Pi_{h}v)\|_{L_2(Q)}^2 + h^2 \sum_{i=1}^d \sum_{j=1}^{d+1} | \partial_{x_i} \partial_{x_j} (v - \Pi_{h}v) |_{H^1(Q)}^2 \right]\\
    &\leq \left[ \sqrt{2 (1 + d(d+1)) C_u}  C_{tr} C_s  h^{s-1} \|v\|_{H^{s}(Q)}\right]^2,
  \end{align*}
  which completes the proof.
\qed
\end{proof}
\begin{lemma}
 \label{lem:consistencyMoving}
  Let $p \geq 2$. If the solution $u \in H^{1,0}_0(Q)$  of the variational problem 
  \eqref{eqn:VariationalFormulation} belongs to $H^2(Q)$,
  then it satisfies the consistency identity
  \begin{equation*}
    b_h(u,v_h) = l_h(v_h) \quad \forall \; v_h \in V_{0h}.
  \end{equation*}
\end{lemma}
\begin{proof}
With the same arguments as in Lemma \ref{lem:consistency2}, we obtain that
\begin{align}\label{eqn:strongL2Form}
 \partial_t u - \Delta u = f \quad \text{in } L_2(Q) \qquad \text{and} \qquad u = u_0 \quad \text{in } L_2(\Sigma_0).
\end{align}
We now multiply the differential equation of \eqref{eqn:strongL2Form} with a test function $v_h + \theta h \partial_t v_h$ for $v_h \in V_{0h}$ and integrate over the space-time domain $Q$. Since
$u \in H^2(Q)$ and $p \geq 2$,
we can apply all the derivations as we did 
at
the beginning of this section to obtain the statement of this Lemma.
\qed
\end{proof}
\begin{theorem}
\label{thm:errorestimatemovingdomain}
  Let $p\geq 2$ and $\theta$ be sufficiently small, see Lemma \ref{lem:movingcoercivity}. 
  Furthermore, let $u \in H^{1,0}_0(Q) \cap H^{s}(Q)$ with $s \geq 2$ 
  be the exact solution of the model problem \eqref{eqn:VariationalFormulation},
  and let $u_h \in V_{0h}$ be the solution to discrete variational problem \eqref{eqn:MovingDiscreteVariationalFormulation}.
  Then the discretization error estimate
  \begin{equation}
  \label{eqn:errorestimatemovingdomain}
    \|u - u_h\|_{h,{m}} \leq Ch^{t- 1} \|u\|_{H^{t}(Q)}, 
  \end{equation}
  holds, where $t = \min\{s,p+1\}$, 
  and $p$ denotes the underlying polynomial degree of the NURBS.
\end{theorem}
\begin{proof}
The stated error estimate follows in the exact way as in Theorem \ref{thm:EnergyNormErrorEstimate} by using Lemma \ref{lem:movingcoercivity}--\ref{lem:consistencyMoving}.
\qed
\end{proof}
%
%
\begin{remark}
 \label{rem:p=1:movingdomain}
 For the case $p=1$, Lemma \ref{lem:consistencyMoving} is not valid in general. 
In this case, we have to take additional consistency errors into account 
by estimating them properly. It turns out, 
cf. Subsection \ref{subsec:MovingSpatialComputationalDomain}, 
that also, for the case $p=1$, we obtain the 
full
convergence rates  wrt the energy norm $\| \cdot \|_{h,m}$
for smooth solutions.
\end{remark}
%
\section{Numerical Results}
\label{NumericalResults}
The numerical results presented below have been performed in \gismo 
\footnote{\textit{Geometry plus Simulation Modules}, http://www.gs.jku.at/gismo} 
\cite{LMN:JuettlerLangerMantzaflarisMooreZulehner:2014a}. We used the sparse direct solver 
SuperLU to solve the resulting linear system \eqref{eqn:algebraicequation} 
of IgA equations. 
We present numerical results for both fixed 
(Subsection~\ref{subsec:FixedSpatialComputationalDomain}) and 
moving (Subsection~\ref{subsec:MovingSpatialComputationalDomain}) spatial 
computational domains in one and two dimensions in space.
In Subsection~\ref{subsec:ParallelSolution}, we present numerical results 
which demonstrate the efficiency of a standard parallel AMG preconditioned 
GMRES solver on massively parallel computers with several thousands
of cores.
We mention that $\theta$ was set to $0.1$ in all our numerical experiments
including the examples with moving spatial domains.
\subsection{Fixed Spatial Computational Domain}
\label{subsec:FixedSpatialComputationalDomain}

\subsubsection{Fixed one-dimensional spatial computational domain.}
\label{subsubsec:Fixedspatialcomputationaldomain1d}

We consider the one-dimensional spatial domain $\Omega =(0,1)$ and the time interval $(0,T)$ 
with $T=1$,
i.e., we have the space-time cylinder $Q = (0,1)^2$,
that can geometrically represented by  the knot vectors 
$\Xi_{1} = \{0, 0, 1, 1\}$ and $\Xi_{2} = \{0, 0, 1, 1\}$ 
in the IgA context. 
We solve our parabolic boundary value problem \eqref{eqn:ModelProblem},
and choose the data such that the solution is given by 
$u(x_1,t) = u(x,t) = \sin(\pi x) \sin(\pi t)$,
i.e. $f(x,t)=\partial_{t} u(x,t) - \Delta u(x,t)= 
\pi \sin(\pi x)( \cos(\pi t) + \pi \sin(\pi t) )$ in $Q$, 
$u_0 = 0$ on $\overline{\Omega}$, and $u$ obviously
vanishes on $\Sigma$. Thus, the compatibility condition between boundary and initial 
conditions holds.
The convergence behavior of the space-time IgA scheme 
with respect to the discrete norm $\| \cdot \|_h$ 
is shown in 
Tables~\ref{table:UnitSquareI:discretenorm} and \ref{table:UnitSquareII:discretenorm} 
by a series of $h$-refinement and by using B-splines of polynomial 
degrees $p = 1,2,3,4$.
After some saturation, we observe the optimal convergence rate $O(h^p)$ 
as  theoretically predicted by Theorem~\ref{thm:EnergyNormErrorEstimate} 
for smooth solutions.
Moreover, 
Tables~\ref{table:UnitSquareI} and \ref{table:UnitSquareII} 
show the $L_2$ errors and the corresponding rates 
for the same setting. We see that the  $L_2$ rates are 
asymptotically optimal as well, i.e. 
the $L_2$-error behaves like $O(h^{p+1})$.
%
\begin{figure}[ht!]
     \begin{center}
       \includegraphics[width = 0.60\textwidth]{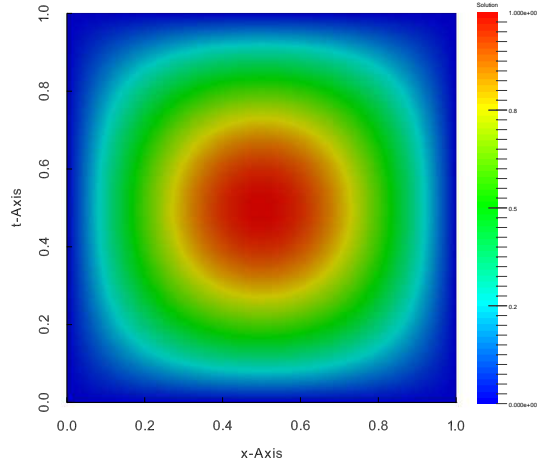}
    \end{center}
    \caption{Solution contours in the space-time cylinder $Q$ for  Example~\ref{subsubsec:Fixedspatialcomputationaldomain1d}.}
    \label{fig:UnitSsquare}
\end{figure}

\begin{table}[htb!]
 \centering
  \begin{tabular}{|l|l|l|l|l|l|l|l|} \hline
    \multicolumn{3}{|c|}{p = 1} &\multicolumn{3}{|c|}{p = 2 } \\
     \cline{1-6}   
       Dofs   &$\|u - u_h \|_{h}$ &Rate   &Dofs     &$\|u - u_h \|_{h}$     &Rate  \\  \hline
         4    &1.6782e+00         &0       & 9       &2.10399e-01            &0	    \\ 
         9    &7.28214e-01        &1.20    &16       &2.03729e-01            &0.04  \\
        25    &3.61278e-01        &1.01    &36       &3.98228e-02            &2.35  \\
        81    &1.79489e-01        &1.01    &100      &9.29436e-03            &2.10  \\
       289    &8.94084e-01        &1.01    &324      &2.27848e-03            &2.02  \\
      1089    &4.46132e-02        &1.00    &1156     &5.66197e-04            &2.01  \\
      4225    &2.22829e-02        &1.00    &4356     &1.41258e-04            &2.00  \\ 
     16641    &1.11354e-02        &1.00    &16900    &3.52865e-05            &2.00  \\ \hline
  \end{tabular}  
\caption{Errors and rates wrt $\| \cdot \|_{h}$ for Example~\ref{subsubsec:Fixedspatialcomputationaldomain1d} 
and degrees $p = 1$ and $p=2$.}
  \label{table:UnitSquareI:discretenorm}
\end{table}
%
\begin{table}[htb!]
 \centering
  \begin{tabular}{|l|l|l|l|l|l|l|l|} \hline
    \multicolumn{3}{|c|}{p = 3} &\multicolumn{3}{|c|}{p = 4 } \\
     \cline{1-6}   
       Dofs   &$\|u - u_h \|_{h}$  &Rate &Dofs  &$\|u - u_h \|_{h}$ &Rate \\  \hline
         16            &2.10106e-01    &0       &25       &6.70957e-03       &0		\\
         25            &2.73234e-02    &2.92    &36       &6.49630e-03       &0.05 	\\
         49            &5.08124e-03    &2.43    &64       &5.56376e-04       &3.55 	\\
        121            &5.73528e-04    &3.15    &144      &3.27142e-05       &4.10  	\\
        361            &6.93807e-05    &3.05    &400      &2.05481e-06       &4.00  	\\
       1225            &8.58843e-06    &3.01    &1296     &1.30057e-07       &4.00   	\\
       4489            &1.07029e-06    &3.00    &4624     &8.20252e-09       &4.00 	\\
       17161           &1.33647e-07    &3.00    &17424    &5.15378e-10       &4.00      \\ \hline
  \end{tabular}  
  \caption{Errors and rates wrt $\| \cdot \|_{h}$ for Example~\ref{subsubsec:Fixedspatialcomputationaldomain1d} 
  and degrees $p = 3$ and $p=4$.}
  \label{table:UnitSquareII:discretenorm}
\end{table}
\begin{table}[htb!]
 \centering
  \begin{tabular}{|l|l|l|l|l|l|l|l|} \hline
    \multicolumn{3}{|c|}{p = 1} &\multicolumn{3}{|c|}{p = 2 } \\
     \cline{1-6}   
       Dofs   &$\|u - u_h \|_{L_2(Q)}$ &Rate    &Dofs     &$\|u - u_h \|_{L_2(Q)}$ &Rate \\  \hline
         4    &5.000e-01               &0       & 9       &2.69186e-02             &0	  \\ 
         9    &1.21333e-01             &2.04    &16       &2.66817e-02             &0.03	  \\
        25    &3.03720e-02             &2.00    &36       &2.30767e-03             &3.53	  \\
        81    &7.47639e-03             &2.02    &100      &2.60187e-04             &3.15 	  \\
       289    &1.85552e-03             &2.01    &324      &3.16609e-05             &3.04    \\
      1089    &4.62271e-04             &2.00    &1156     &3.92785e-06             &3.01    \\
      4225    &1.15377e-04             &2.00    &4356     &4.89712e-07             &3.01    \\ 
     16641    &2.88212e-05             &2.00    &16900    &6.11484e-08             &3.00    \\ \hline
  \end{tabular}  
  \caption{$L_2$ errors and rates for Example~\ref{subsubsec:Fixedspatialcomputationaldomain1d} 
             using degrees $p = 1$ and $p=2$.}
  \label{table:UnitSquareI}
\end{table}
%
\begin{table}[htb!]
 \centering
  \begin{tabular}{|l|l|l|l|l|l|l|l|} \hline
    \multicolumn{3}{|c|}{p = 3} &\multicolumn{3}{|c|}{p = 4 } \\
     \cline{1-6}   
       Dofs   &$\|u - u_h \|_{L_2(Q)}$  &Rate    &Dofs     &$\|u - u_h \|_{L_2(Q)}$ &Rate \\  \hline
         16            &2.65068e-02     &0       &25       &5.49027e-04       &0		\\
         25            &2.32514e-03     &3.16    &36       &5.43598e-04       &0.01 	\\
         49            &3.10354e-04     &2.84    &64       &3.90240e-05       &3.80 	\\
        121            &1.63884e-05     &4.25    &144      &1.02403e-06       &5.25  	\\
        361            &9.72892e-07     &4.10    &400      &3.03989e-08       &5.07  	\\
       1225            &5.99946e-08     &4.02    &1296     &9.38131e-10       &5.02   	\\
       4489            &3.73705e-09     &4.01    &4624     &2.92231e-11       &5.00 	\\
       17161           &2.33370e-10     &4.00    &17424    &9.15973e-13       &5.00      \\ \hline
  \end{tabular}  
  \caption{$L_2$ errors and rates for Example~\ref{subsubsec:Fixedspatialcomputationaldomain1d} 
             using degrees $p = 3$ and $p=4$.}
  \label{table:UnitSquareII}
\end{table}
\subsubsection{Fixed two-dimensional spatial computational domain.}
\label{subsubsec:Fixedspatialcomputationaldomain2d}
As a second example we consider  the two-dimensional spatial domain $\Omega =(0,1)^2$ and the time interval $(0,T)$ 
with $T=1$, i.e., we have the space-time cylinder $Q = (0,1)^3$, 
that can geometrically be represented by the knot vectors 
$\Xi_{1} = \{0, 0, 1, 1\},$ $\Xi_{2} = \{0, 0, 1, 1\}$ and $\Xi_{3} = \{0, 0, 1, 1\}$ 
in the context of IgA. 
We solve our model problem \eqref{eqn:ModelProblem},
and again choose the data such that the solution is given by 
$u(x,t) = \sin(\pi x_1)\sin(\pi x_2)  \sin(\pi t)$,
i.e.
$f(x,t)=\partial_{t} u(x,t) - \Delta u(x,t) 
=  \pi  \sin(\pi x_1)\sin(\pi x_2) (\cos(\pi t) + 2\pi \sin(\pi t) )$ in $Q$, 
$u_0 = 0$ on $\overline{\Omega}$, and $u$ obviously
vanishes on $\Sigma$. Thus, the compatibility condition between boundary and initial 
conditions holds.
In Figure~\ref{fig:UnitCube}, we present the solution contours of the problem in $\mathbb{R}^3,$
where we have sliced the domain along the $t$-axis.
The convergence behavior of the space-time IgA scheme with respect to the discrete norm $\| \cdot \|_h$ 
is shown in Tables~\ref{table:UnitCubeI:discrete} and \ref{table:UnitCubeII:discrete}
by a series of $h$-refinement and by using B-splines of polynomial 
degrees $p = 1,2,3,4$.
After some saturation, we observe the optimal convergence rate $O(h^p)$ 
as  theoretically predicted by Theorem~\ref{thm:EnergyNormErrorEstimate} 
for smooth solutions.
Moreover, Tables~\ref{table:UnitCubeI} and \ref{table:UnitCubeII} show the $L_2$ errors and the corresponding rates 
for the same setting. We see that the  $L_2$ rates are asymptotically optimal as well, i.e. 
the $L_2$-error behaves like $O(h^{p+1})$.

\begin{figure}[ht!]
     \begin{center}
       \includegraphics[width = 0.70\textwidth]{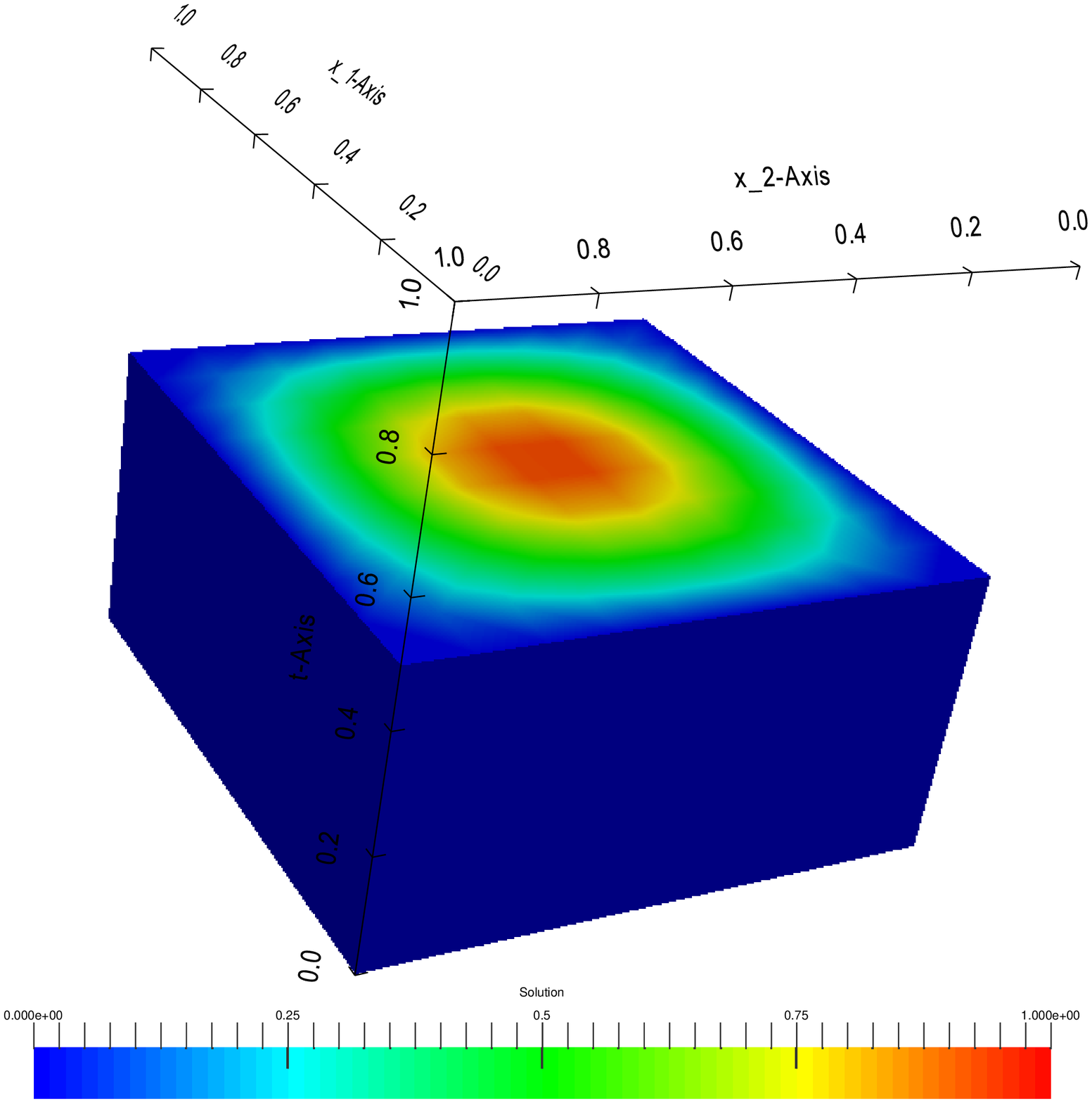}
    \end{center}
    \caption{Solution contours in the space-time cylinder $Q$ at 
	    $t = 0.5$
	    for Example~\ref{subsubsec:Fixedspatialcomputationaldomain2d}. }\label{fig:UnitCube}
\end{figure}
\begin{table}[htb!]
 \centering
  \begin{tabular}{|l|l|l|l|l|l|l|l|} \hline
    \multicolumn{3}{|c|}{p = 1} &\multicolumn{3}{|c|}{p = 2 } \\
     \cline{1-6}   
       Dofs        &$\|u - u_h \|_h$   &Rate    &Dofs  &$\|u - u_h \|_h$ &Rate     \\  \hline
         8         &1.63740            &0       &27    &2.15440e-01      &0        \\
        27         &7.39981e-01        &1.15    &64    &2.11247e-01      &0.03     \\
       125         &3.60495e-01        &1.04    &216   &3.98871e-02      &2.40     \\
      729          &1.79065e-01        &1.01    &1000  &9.27926e-03      &2.10     \\
      4913         &8.92787e-02        &1.00    &5832  &2.27556e-03      &2.03     \\
     35937         &4.45779e-02        &1.00    &39304 &5.65772e-04      &2.01     \\  \hline
  \end{tabular}  
 \caption{Errors and rates wrt $\| \cdot \|_{h}$ for Example~\ref{subsubsec:Fixedspatialcomputationaldomain2d} 
  and degrees $p = 1$ and $p=2$.}\label{table:UnitCubeI:discrete}
\end{table}

\begin{table}[htb!]
 \centering
  \begin{tabular}{|l|l|l|l|l|l|l|l|} \hline
    \multicolumn{3}{|c|}{p = 3} &\multicolumn{3}{|c|}{p = 4 } \\
     \cline{1-6}   
       Dofs   &$\|u - u_h \|_h$  &Rate &Dofs       &$\|u - u_h \|_h$ &Rate \\  \hline
         64           &2.16883e-01     &0          &125          &6.67416e-03             &0          \\
        125           &2.75120e-02     &2.97       &216          &7.69213e-03             &0.03       \\
        343           &5.09465e-03     &2.43       &512          &3.55820e-03             &3.55       \\
       1331           &5.72742e-04     &3.15       &1728         &3.26623e-05             &4.10       \\
       6859           &6.92964e-05     &3.04       &8000         &2.05215e-06             &4.00       \\ 
      42875           &8.58214e-06     &3.01       &46656        &1.37611e-07             &4.00       \\  \hline
  \end{tabular}  
 \caption{Errors and rates wrt $\| \cdot \|_{h}$ for Example~\ref{subsubsec:Fixedspatialcomputationaldomain2d} 
  and degrees $p = 3$ and $p=4$.}\label{table:UnitCubeII:discrete}
\end{table}

\begin{table}[htb!]
 \centering
  \begin{tabular}{|l|l|l|l|l|l|l|l|} \hline
    \multicolumn{3}{|c|}{p = 1} &\multicolumn{3}{|c|}{p = 2 } \\
     \cline{1-6}   
       Dofs   &$\|u - u_h \|_{L_2(Q)}$ &Rate &Dofs     &$\|u - u_h \|_{L_2(Q)}$ &Rate  \\  \hline
         8         &3.65528e-01        &0       &27    &2.39153e-02     &0             \\
        27         &9.56396e-02        &1.93    &64    &2.37388e-02     &0.01          \\
       125         &2.32679e-02        &2.03    &216   &1.99848e-03     &3.57          \\
      729          &5.75358e-03        &2.01    &1000  &2.22710e-04     &3.17          \\
      4913         &1.43171e-03        &2.01    &5832  &2.70486e-05     &3.04          \\
     35937         &3.57195e-04        &2.00    &39304 &3.35780e-06     &3.00          \\  \hline
  \end{tabular}  
 \caption{$L_2$ errors and rates for Example~\ref{subsubsec:Fixedspatialcomputationaldomain2d}
          using degree $p = 1$ and $p=2.$}\label{table:UnitCubeI}
\end{table}

\begin{table}[htb!]
 \centering
  \begin{tabular}{|l|l|l|l|l|l|l|l|} \hline
    \multicolumn{3}{|c|}{p = 3} &\multicolumn{3}{|c|}{p = 4 } \\
     \cline{1-6}   
       Dofs   &$\|u - u_h \|_{L_2(Q)}$  &Rate   &Dofs         &$\|u - u_h \|_{L_2(Q)}$ &Rate \\  \hline
         64           &2.39325e-02      &0      &125          &4.75391e-04             &0             \\
        125           &2.03108e-03      &3.56   &216          &4.73425e-04             &0.01       \\
        343           &2.68174e-04      &2.92   &512          &3.34666e-05             &3.82       \\
       1331           &1.41715e-05      &4.24   &1728         &8.74291e-07             &5.26       \\
       6859           &8.42223e-07      &4.07   &8000         &2.60544e-08             &5.07       \\  
      42875           &5.19528e-08      &4.01   &46656        &8.07120e-10             &5.01       \\ \hline
  \end{tabular}  
 \caption{$L_2$ errors and rates for Example~\ref{subsubsec:Fixedspatialcomputationaldomain2d}
            using degree $p = 3$ and $p = 4.$ }\label{table:UnitCubeII}
\end{table}

\subsection{Moving Spatial Computational Domain}
\label{subsec:MovingSpatialComputationalDomain}

\subsubsection{A simple  one-dimensional moving spatial computational domain.}
\label{subsubsec:Simplemovingspatialcomputationaldomain2D}
%

Now we consider the one-dimensional moving computational domain 
$\Omega(t) = \{x=x_1 \in \mathbb{R}^1: a(t) < x < b(t)\}$ 
with $t = (0,T)$,
where $a(t) = -t/2$, $b(t)= 1+t/2$, and $T=1$. 
The space-time cylinder 
$Q = \{(x,t) \in \mathbb{R}^2:\; x \in \Omega(t) ,t \in (0,T)\} \subset \mathbb{R}^{2}$
is obviously a fixed domain in the space-time world $\mathbb{R}^{2}$. 
It can geometrically be represented by the knot vectors 
$\Xi_{1} = \{0, 0, 1, 1\}$ and $\Xi_{2} = \{0, 0, 1, 1\}$ 
and the control points 
$\mathbf{P}_{1,1} = (0,0)$,
$\mathbf{P}_{2,1} = (1,0)$,
$\mathbf{P}_{2,2} = (1.5,1.0)$,
and
$\mathbf{P}_{1,2} = (-0.5,1.0)$
in the context of IgA. 
We solve our model problem \eqref{eqn:ModelProblem},
and again choose the data such that the solution is given by $u(x,t) = \sin(\pi x)  \sin(\pi t)$,
i.e. $f(x,t)=\partial_{t} u(x,t) - \Delta u(x,t)= ( \pi  \sin(\pi x)) (\cos(\pi t) + \pi \sin(\pi t) )$ in $Q$, 
$u_0 = 0$ on $\overline{\Omega}$,
and $u(x,t) = \sin(\pi x)  \sin(\pi t)$ on $\Sigma$.
Thus, the compatibility condition between boundary and initial conditions holds.
The space-time computational domain $Q$ and the solution is drawn in Figure~\ref{fig:Moving2DLinear}.
The convergence behavior of the space-time IgA scheme with respect to the discrete norm $\| \cdot \|_{h,m}$ 
is shown in Tables~\ref{table:SimpleMovingI:discrete} and \ref{table:SimpleMovingII:discrete}
by a series of $h$-refinement and by using B-splines of polynomial 
degrees $p = 1,2,3,4$.
After some saturation, we observe the optimal convergence rate $O(h^p)$ for  $p \geq	 2$ 
as theoretically predicted by  Theorem~\ref{thm:errorestimatemovingdomain}
for smooth solutions.
Moreover, Tables~\ref{table:UnitCubeI} and \ref{table:UnitCubeII} show the $L_2$ errors and the corresponding rates 
for the same setting. We see that the  $L_2$ rates are asymptotically optimal for $p\ge2$ as well, 
i.e. they behave like $O(h^{p+1})$.
For  $p=1$, we also  observe the optimal rate in the discrete norm,
cf. also Remark~ \ref{rem:p=1:movingdomain}, 
whereas the $L_2$-rate does not reach the optimal order 2. 
\begin{figure}[htb!]
     \begin{center}
       \includegraphics[width = 0.7\textwidth, height = 0.35\textheight]{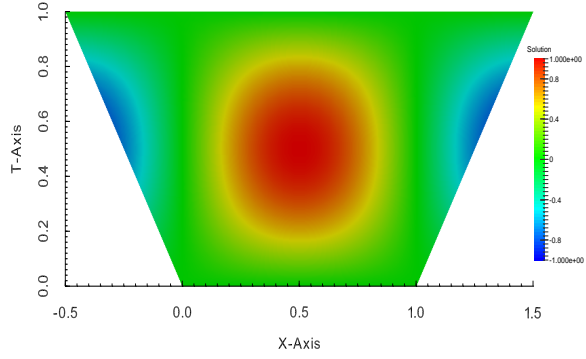}
    \end{center}
    \caption{Solution contours in the space-time cylinder $Q$ for Example~\ref{subsubsec:Simplemovingspatialcomputationaldomain2D}.}
    \label{fig:Moving2DLinear}
\end{figure}

\begin{table}[htb!]
 \centering
  \begin{tabular}{|l|l|l|l|l|l|l|l|} \hline
    \multicolumn{3}{|c|}{p = 1} &\multicolumn{3}{|c|}{p = 2 } \\
     \cline{1-6}   
       Dofs           &$\| u - u_h  \|_{h,m}$&Rate &Dofs   &$\| u - u_h  \|_{h,m}$ &Rate \\  \hline
         4            &2.31256       &0          &9           &8.3047e-01    &0     \\
         9            &8.7222e-01    &1.41       &16          &7.4075e-01    &0.17  \\
         25           &5.9385e-01    &0.55       &36          &1.3091e-01    &2.50  \\
         81           &3.0387e-01    &0.97       &100         &2.9977e-02    &2.13  \\
        289           &1.5240e-01    &1.00       &324         &7.2872e-03    &2.04  \\
       1089           &7.6181e-02    &1.00       &1156        &1.8068e-03    &2.01  \\
       4225           &3.8070e-02    &1.00       &4356        &4.5056e-04    &2.00  \\ 
      16641           &1.9028e-02    &1.00       &16900       &1.1255e-04    &2.00  \\ \hline
  \end{tabular}  
  \caption{Errors and rates wrt $\| \cdot \|_{h,m}$ for Example~\ref{subsubsec:Simplemovingspatialcomputationaldomain2D} 
  and degrees $p = 1$ and $p=2$.}\label{table:SimpleMovingI:discrete}
\end{table}
\begin{table}[htb!]
 \centering
  \begin{tabular}{|l|l|l|l|l|l|l|l|} \hline
    \multicolumn{3}{|c|}{p = 3} &\multicolumn{3}{|c|}{p = 4 } \\
     \cline{1-6}   
       Dofs   &$\|u - u_h  \|_{h,m}$  &Rate       &Dofs         &$\| u - u_h  \|_{h,m}$ &Rate \\  \hline
         16           &7.87827e-01      &0          &25           &9.58143e-02   &0    \\
         25           &7.74683e-02      &3.35       &36           &7.29388e-02   &0.35 \\
         49           &2.68888e-02      &1.53       &64           &7.25760e-03   &3.31 \\
        121           &2.54304e-03      &3.40       &144          &3.25434e-04   &4.48\\
        361           &2.97843e-04      &3.09       &400          &1.83582e-05   &4.15\\
       1225           &3.69216e-05      &3.01       &1296         &1.11825e-06   &4.04 \\
       4489           &4.62600e-06      &3.00       &4624         &6.95208e-08   &4.01\\ 
      17161           &5.79880e-07      &3.00       &17424        &4.34224e-09   &4.00 \\ \hline
  \end{tabular}  
 \caption{Errors and rates wrt $\| \cdot \|_{h,m}$ for Example~\ref{subsubsec:Simplemovingspatialcomputationaldomain2D}
            using degree $p = 3$ and $p=4.$}\label{table:SimpleMovingII:discrete}
\end{table}
\begin{table}[htb!]
 \centering
  \begin{tabular}{|l|l|l|l|l|l|l|l|} \hline
    \multicolumn{3}{|c|}{p = 1} &\multicolumn{3}{|c|}{p = 2 } \\
     \cline{1-6}   
       Dofs   &$\|u - u_h \|_{L_2(Q)}$  &Rate &Dofs   &$\|u - u_h \|_{L_2(Q)}$ &Rate \\  \hline
         4            &8.7462          &0          &9           &1.9450e-01     &0     \\
         9            &1.9970e-01      &2.13       &16          &1.41565e-01    &0.46  \\
         25           &7.3226e-02      &1.45       &36          &1.09044e-02    &3.70  \\
         81           &1.8794e-02      &1.96       &100         &1.12045e-03    &3.28  \\
        289           &5.0344e-03      &1.90       &324         &1.31179e-04    &3.09  \\
       1089           &1.5998e-03      &1.65       &1156        &1.60564e-05    &3.03  \\
       4225           &6.4746e-04      &1.31       &4356        &1.99240e-06    &3.01  \\ 
      16641           &3.0437e-04      &1.10       &16900       &2.48337e-07    &3.00  \\ \hline
  \end{tabular}  
 \caption{$L_2$ errors and rates for Example~\ref{subsubsec:Simplemovingspatialcomputationaldomain2D}
            using degree $p = 1$ and $p=2.$}\label{table:SimpleMovingI}
\end{table}
\begin{table}[htb!]
 \centering
  \begin{tabular}{|l|l|l|l|l|l|l|l|} \hline
    \multicolumn{3}{|c|}{p = 3} &\multicolumn{3}{|c|}{p = 4 } \\
     \cline{1-6}   
       Dofs   &$\|u - u_h \|_{L_2(Q)}$   &Rate       &Dofs   &$\|u - u_h \|_{L_2(Q)}$ &Rate \\  \hline
         16           & 1.69684e-01      &0          &25           &1.38383e-02     &0    \\
         25           & 1.02458e-02      &4.05       &36           &1.01827e-02     &0.44 \\
         49           & 2.34449e-03      &2.13       &64           &6.76183e-04     &3.91 \\
        121           & 9.90173e-05      &4.57       &144          &1.2387e-05      &5.77 \\
        361           & 5.56308e-06      &4.15       &400          &3.23506e-07     &5.26 \\
       1225           & 3.40489e-07      &4.03       &1296         &9.61797e-09     &5.07 \\
       4489           & 2.12529e-08      &4.00       &4624         &2.96841e-10     &5.02 \\ 
      17161           & 1.33061e-09      &4.00       &17424        &9.24929e-12     &5.00 \\ \hline
  \end{tabular}  
 \caption{$L_2$ errors and rates for Example~\ref{subsubsec:Simplemovingspatialcomputationaldomain2D}
            using degree $p = 3$ and $p=4.$}\label{table:SimpleMovingII}
\end{table}

\subsubsection{Curvilinearly  moving one-dimensional  spatial computational domain.}
\label{subsubsec:Curvilinearmovingspatialcomputationaldomain1D}
We consider again an
one-dimensional  moving spatial computational domain of the form
$\Omega(t) = \{x=x_1 \in \mathbb{R}^1: a(t) < x < b(t)\}$, 
$t \in (0,1)$,
where now the movement is described by the 
functions $a(t) = t(1 -t)/2$ and $b(t)=1-t(1 -t)/2 $ 
leading to the space-time cylinder 
$Q = \{(x,t) \in \mathbb{R}^2:\; x \in \Omega(t) ,t \in (0,T)\} \subset \mathbb{R}^{2}$
with a curved surface area $\Sigma$, see Figure~\ref{fig:Moving2D}.
The space-time cylinder $Q$ can also be represented by the knot vectors 
$\Xi_{1} = \{0, 0, 1, 1\}$ and $\Xi_{2} = \{0, 0, 0, 1, 1, 1\}$ 
and the corresponding control points 
$\mathbf{P}_{1,1} = (0,0)$,
$\mathbf{P}_{2,1} = (1,0)$,
$\mathbf{P}_{2,2} = (0.75,0.5)$,
$\mathbf{P}_{2,3} = (1,1)$,
$\mathbf{P}_{1,3} = (0,1)$
and
$\mathbf{P}_{1,2} = (0.25,0.5)$
in the context of IgA, see also  Figure~\ref{fig:Moving2D}.
We solve our model problem \eqref{eqn:ModelProblem},
and again choose the data such that the solution is given by $u(x,t) = \sin(\pi x)  \sin(\pi t)$,
i.e. $f(x,t)=\partial_{t} u(x,t) - \Delta u(x,t)= ( \pi  \sin(\pi x)) (\cos(\pi t) + \pi \sin(\pi t) )$ in $Q$, 
$u_0 = 0$ on $\overline{\Omega}$,  and $u(x,t)  = \sin(\pi x)  \sin(\pi t)$ on $\Sigma$. 
Thus, the compatibility condition between boundary and initial 
conditions hold.
The convergence behavior of the space-time IgA scheme with respect to the discrete norm $\| \cdot \|_{h,m}$ 
is shown in Tables~\ref{table:CurvilinearI:discrete} and \ref{table:CurvilinearII:discrete}
by a series of $h$-refinement and by using B-splines of polynomial 
degrees $p = 1,2,3,4$.
After some saturation, we observe the optimal convergence rate $O(h^p)$ for $p \ge 2$ 
as theoretically predicted by Theorem~\ref{thm:errorestimatemovingdomain} 
for smooth solutions.
Moreover, Tables~\ref{table:CurvilinearI} and \ref{table:CurvilinearII} show the $L_2$ errors and the corresponding rates 
for the same setting. We see that the  $L_2$ rates are asymptotically optimal for $p\ge 2$ as well, 
i.e. they behave like $O(h^{p+1})$.
For  $p=1$, we also  observe the optimal rate in the discrete norm,
cf. also Remark~\ref{rem:p=1:movingdomain}, 
whereas the $L_2$-rate does not reach the optimal order 2. 

\begin{figure}[htb!]
     \begin{center}
      \includegraphics[width  = 0.46\textwidth, height = 0.26\textheight]{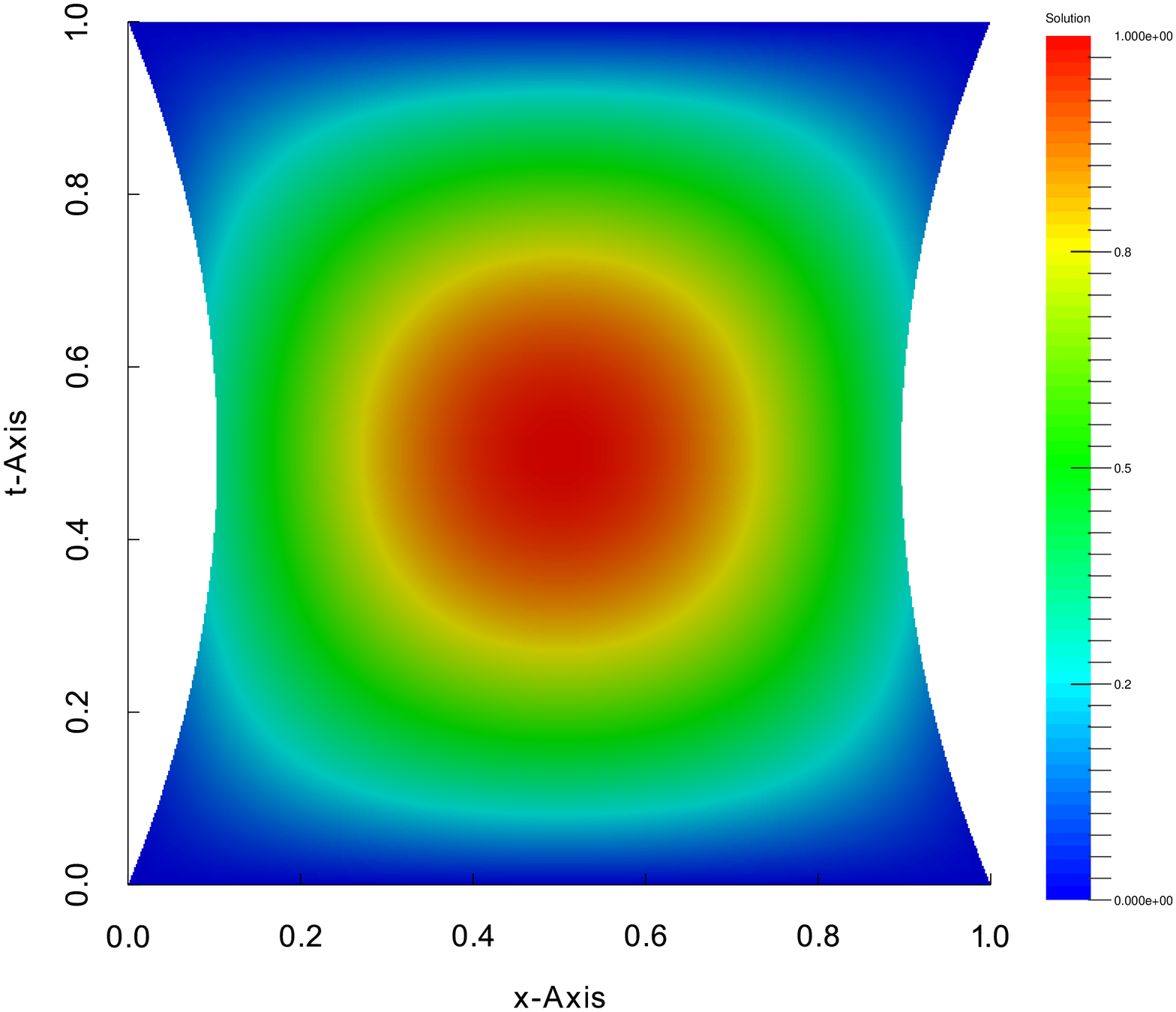}
      \hspace{8.00mm}
       \includegraphics[width = 0.45\textwidth,  height = 0.28\textheight]{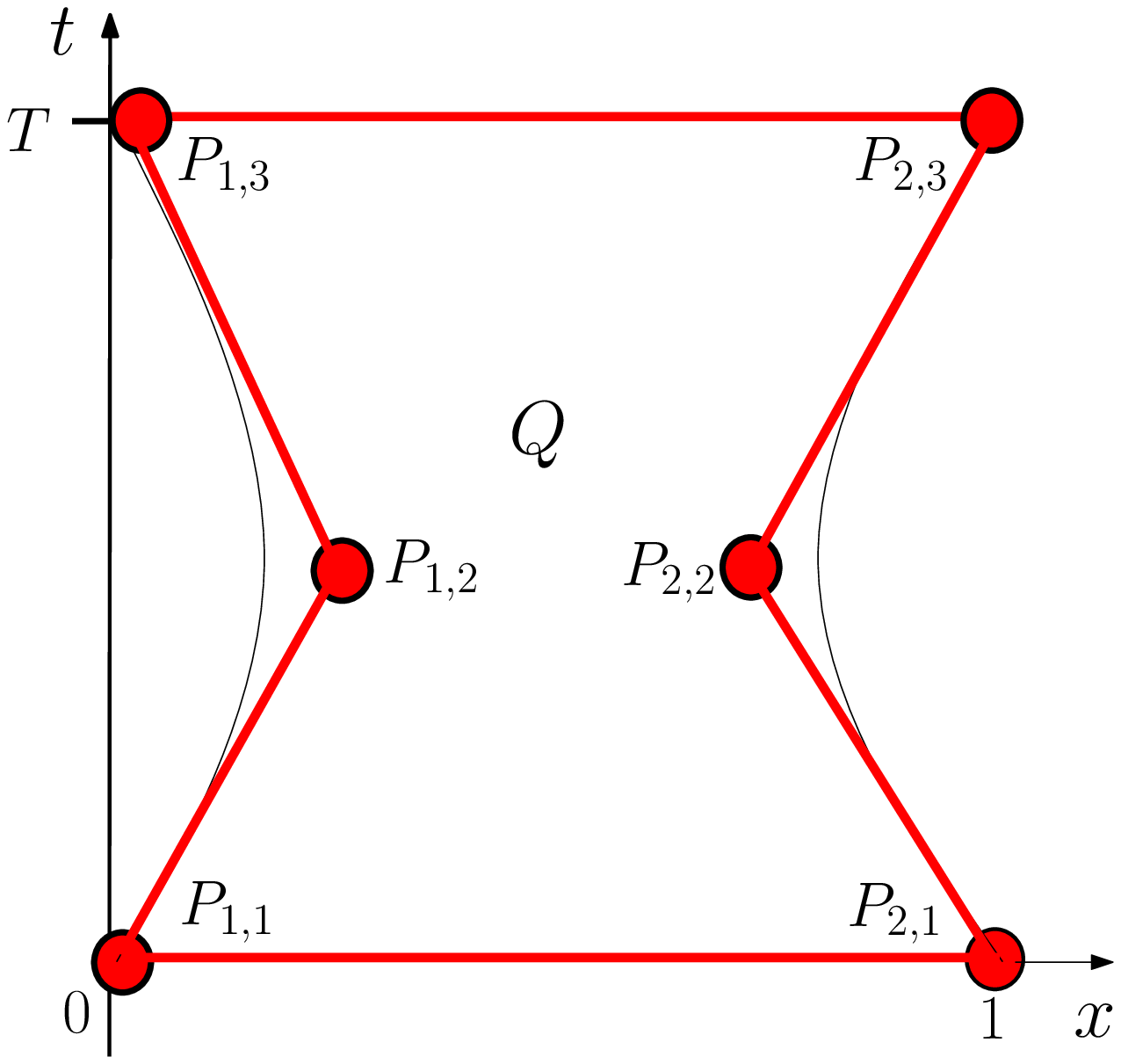}
    \end{center}
    \caption{Solution contours 
    for 
    Example~\ref{subsubsec:Curvilinearmovingspatialcomputationaldomain1D} (left) and the control points (right).} 
      \label{fig:Moving2D}
\end{figure}
\begin{table}[htb!]
 \centering
  \begin{tabular}{|l|l|l|l|l|l|l|l|} \hline
    \multicolumn{3}{|c|}{p = 1} &\multicolumn{3}{|c|}{p = 2 } \\
     \cline{1-6}   
       Dofs   &$\| u - u_h \|_{h,m}$&Rate &Dofs   &$\| u - u_h  \|_{h,m}$ &Rate \\  \hline
         4            &1.40608e+00   &0          &9           &1.79646e-01    &0                      \\
         9            &6.28191e-01   &1.15       &16          &1.66950e-01    &0.11  \\
         25           &3.02013e-01   &1.05       &36          &2.52637e-02    &2.72  \\
         81           &1.47951e-01   &1.02       &100         &5.45976e-03    &2.21  \\
        289           &7.33135e-02   &1.01       &324         &1.29401e-03    &2.08  \\
       1089           &3.65079e-02   &1.01       &1156        & 3.16214-04    &2.03  \\
       4225           &1.82188e-02   &1.00       &4356        & 7.82250e-05    &2.02  \\ 
      16641           &9.10085e-03   &1.00       &16900       & 1.94575e-05    &2.01  \\ \hline
  \end{tabular}  
  \caption{Errors and rates wrt $\| \cdot \|_{h,m}$ for 
          Example~\ref{subsubsec:Curvilinearmovingspatialcomputationaldomain1D} 
          and degrees $p = 1$ and $p=2$.}\label{table:CurvilinearI:discrete}
\end{table}
\begin{table}[htb!]
 \centering
  \begin{tabular}{|l|l|l|l|l|l|l|l|} \hline
    \multicolumn{3}{|c|}{p = 3} &\multicolumn{3}{|c|}{p = 4 } \\
     \cline{1-6}   
       Dofs   &$\| u - u_h \|_{h,m}$    &Rate       &Dofs       &$\|u - u_h  \|_{h,m}$ &Rate \\  \hline
         16           &2.25659e-01      &0          &25           &1.55287e-02         &0    \\
         25           &1.95569e-02      &3.18       &36           &9.35760e-03         &0.73 \\
         49           &3.37111e-02      &2.56       &64           &3.74209e-04         &4.64 \\
        121           &3.37580e-03      &3.25       &144          &1.93936e-05         &4.27\\
        361           &3.95998e-04      &3.09       &400          &9.83348e-07         &4.30\\
       1225           &4.83941e-06      &3.03       &1296         &5.54882e-08         &4.14 \\
       4489           &5.99023e-07      &3.01       &4624         &3.30365e-09         &4.07\\ 
      17161           &7.45335e-08      &3.01       &17424        &2.01770e-10         &4.03 \\ \hline
  \end{tabular}  
 \caption{Errors and rates wrt $\| \cdot \|_{h,m}$ for 
          Example~\ref{subsubsec:Curvilinearmovingspatialcomputationaldomain1D}
            using degree $p = 3$ and $p=4.$}\label{table:CurvilinearII:discrete}
\end{table}
\begin{table}[htb!]
 \centering
  \begin{tabular}{|l|l|l|l|l|l|l|l|} \hline
    \multicolumn{3}{|c|}{p = 1} &\multicolumn{3}{|c|}{p = 2 } \\
     \cline{1-6}   
       Dofs   &$\|u - u_h \|_{L_2(Q)}$&Rate      &Dofs   &$\|u - u_h \|_{L_2(Q)}$ &Rate \\  \hline
         4            &4.03831e-01   &0          &9           &2.93862e-02    &0                      \\
         9            &9.12333e-02   &2.14       &16          &3.21213e-02    &0.13  \\
         25           &2.33973e-02   &1.96       &36          &2.74769e-03    &3.55  \\
         81           &5.79121e-03   &2.01       &100         &3.10332e-04    &3.15  \\
        289           &1.52473e-03   &1.92       &324         &3.77756e-05    &3.03  \\
       1089           &4.68079e-04   &1.70       &1156        &4.68649e-06    &3.01  \\
       4290           &1.82024e-04   &1.36       &4356        &5.84132e-07    &3.00  \\ 
      16770           &8.37613e-05   &1.12       &16900       &7.29169e-08    &3.00  \\ \hline
  \end{tabular}  
 \caption{$L_2$ errors and rates for Example~\ref{subsubsec:Curvilinearmovingspatialcomputationaldomain1D}
            using degree $p = 1$ and $p=2.$}\label{table:CurvilinearI}
\end{table}
\begin{table}[htb!]
 \centering
  \begin{tabular}{|l|l|l|l|l|l|l|l|} \hline
    \multicolumn{3}{|c|}{p = 3} &\multicolumn{3}{|c|}{p = 4 } \\
     \cline{1-6}   
       Dofs   &$\|u - u_h \|_{L_2(Q)}$   &Rate       &Dofs         &$\|u - u_h \|_{L_2(Q)}$ &Rate \\  \hline
         16           & 3.22870e-02      &0          &25           &1.00347e-03    &0   \\
         25           & 2.61893e-03      &3.62       &36           &9.49495e-04   &0.10 \\
         49           & 4.05805e-04      &2.69       &64           &5.65065e-05   &4.10 \\
        121           & 1.99434e-05      &4.34       &144          &1.96066e-06   &4.85 \\
        361           & 1.15336e-06      &4.11       &400          &6.03483e-08   &5.02 \\
       1225           & 7.10583e-08      &4.02       &1296         &1.88123e-09   &5.00 \\
       4489           & 4.43879e-09      &4.00       &4624         &5.88213e-11   &5.00 \\ 
      17161           & 2.77796e-10      &4.00       &17424        &1.83916e-12   &5.00 \\ \hline
  \end{tabular}  
 \caption{$L_2$ errors and rates for Example~\ref{subsubsec:Curvilinearmovingspatialcomputationaldomain1D}
            using degree $p = 3$ and $p=4.$}\label{table:CurvilinearII}
\end{table}

\subsubsection{Curvilinearly  moving two-dimensional  spatial computational domain.}
\label{subsubsec:Curvilinearmovingspatialcomputationaldomain2D}
We now consider the curvilinearly  moving two-dimensional  spatial computational domain 
$\Omega(t) = \{x=(x_1,x_2)  \in \mathbb{R}^2: a(t) < x_1 < b(t), 0 < x_2 < 1 \}$ 
in $\mathbb{R}^2$, 
where $a(t) = t(1 -t)/2$, $b(t)=1-t(1 -t)/2$, and $t$ is running 
from $0$ to $T=1$, 
%
leading to the space-time cylinder 
$Q = \{(x,t) \in \mathbb{R}^3: \; x \in \Omega(t) ,t \in (0,T)\} \subset \mathbb{R}^{3}$
 that is fixed in the space-time world $\mathbb{R}^{3}$. 
In the context of IgA, $Q$ can geometrically be represented by the knot vectors 
$\Xi_{1} = \{0, 0, 1, 1\}$, $\Xi_{2} = \{0, 0, 1, 1\}$ and $\Xi_{3} = \{0, 0, 0, 1, 1, 1\}$,
and the control points given in Figure~\ref{fig:Moving3D} (right). 
We solve our model problem \eqref{eqn:ModelProblem},
and again choose the data such that the solution is given by 
$u(x,t) = \sin(\pi x_1) \sin(\pi x_2)  \sin(\pi t)$,
i.e. $f(x,t)=\partial_{t} u(x,t) - \Delta u(x,t)=  ( \pi  \sin(\pi x_1)\sin(\pi x_2)) (\cos(\pi t) + 2\pi \sin(\pi t) )$ in $Q$, $u_0 = 0$ on $\overline{\Omega}$, and $u(x,t) = \sin(\pi x_1) \sin(\pi x_2)  \sin(\pi t)$ on $\Sigma$. 
In Figure~\ref{fig:Moving3D} (left), we present the solution contours of the problem in $\mathbb{R}^3$
at $t= 0.5$. 
The convergence behavior of the space-time IgA scheme with respect to the discrete norm $\| \cdot \|_{h,m}$ 
is shown in Tables~\ref{table:Curvilinear2DI:discrete} and \ref{table:Curvilinear2DII:discrete}
by a series of $h$-refinement and by using B-splines of polynomial degrees $p = 1,2,3,4$. 
After some saturation, we observe the optimal convergence rate $O(h^p)$ for $p \ge 2$ 
as theoretically predicted by Theorem~\ref{thm:errorestimatemovingdomain} for smooth solutions.
Moreover, Tables~\ref{table:Curvilinear2DI} and \ref{table:Curvilinear2DII} show the $L_2$ errors and the corresponding rates 
for the same setting. We see that the  $L_2$ rates are asymptotically optimal for $p\ge 2$ as well, i.e. they behave like $O(h^{p+1})$.
%
For  $p=1$, we also  observe the optimal rate in the discrete norm,
cf. also Remark~ \ref{rem:p=1:movingdomain}, 
whereas the $L_2$-rate does not reach the optimal order 2. 
\begin{figure}[htb!]
     \begin{center}
       \includegraphics[width = 0.45\textwidth]{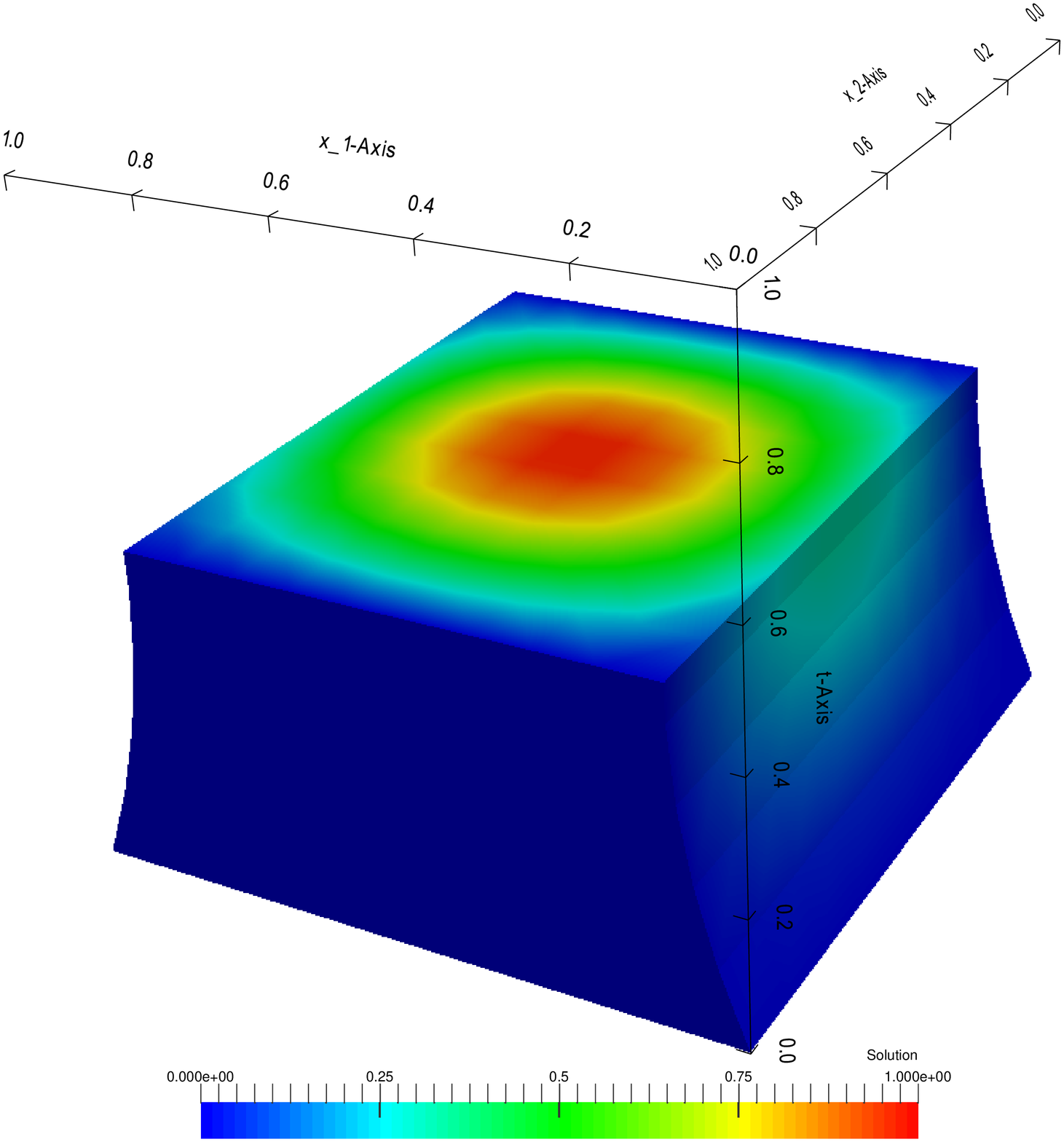}
       \hspace{8.00mm}
       \includegraphics[width = 0.43\textwidth,height = 0.29\textheight]{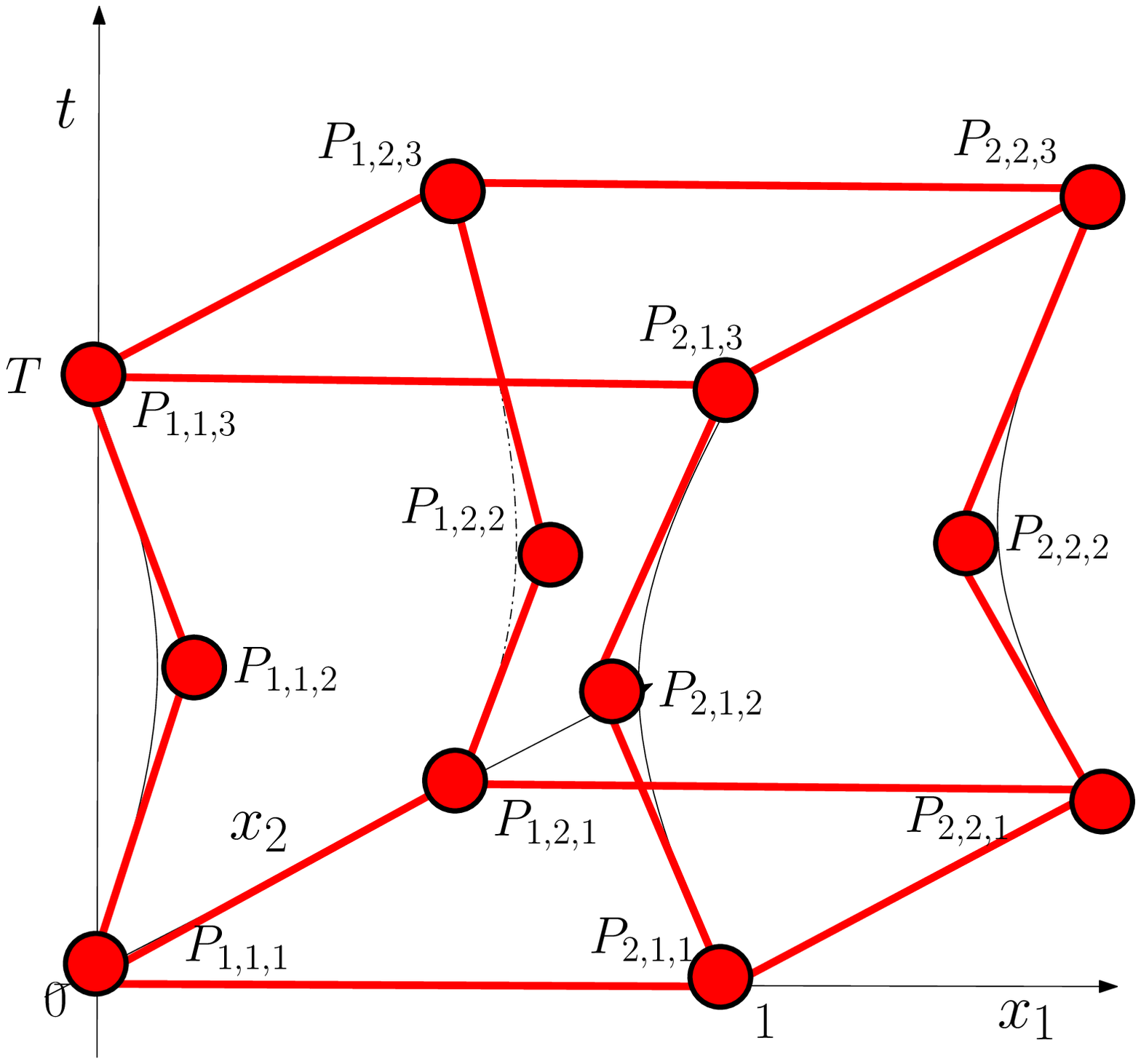}
    \end{center}
   \caption{Solution contours 
    for 
    Example~\ref{subsubsec:Curvilinearmovingspatialcomputationaldomain2D} (left) and the control points (right) 
      are given by
      $\mathbf{P}_{i_1,i_2, i_3} = \{$ $(0,0,0),$ $(1,0,0),$ $(0,1,0),$ $(1,1,0),$ $(0.25,0,0.50),$ $(0.75,0,0.5),$
      $(0.25,1,0.5),$ $(0.75,1,0.5),$ $(0,0,1),$ $(1,0,1),$ $(0,1,1),$ $(1,1,1)$\} 
      for $i_1 = 1,2$, $i_2 = 1,2$ and $i_3 = 1,2,3.$}
      \label{fig:Moving3D}
\end{figure}

\begin{table}[htb!]
 \centering
  \begin{tabular}{|l|l|l|l|l|l|l|l|} \hline
    \multicolumn{3}{|c|}{p = 1} &\multicolumn{3}{|c|}{p = 2 } \\
     \cline{1-6}   
       Dofs   &$\|u - u_h \|_{h,m}$    &Rate &Dofs   &$\|u - u_h \|_{h,m}$   &Rate  \\  \hline
         8          &1.57742e+00     &0          &27           &2.5829e-01   &0     \\
        27          &6.84322e-01     &1.20       &64           &2.0891e-01   &0.31  \\
       125          &3.25869e-01     &1.07       &216          &3.64503e-02  &2.52  \\
       729          &1.60291e-01     &1.02       &1000         &7.88416e-03  &2.21  \\
      4913          &7.96709e-02     &1.01       &5832         &1.85329e-03  &2.08  \\
     35937          &3.97383e-02     &1.00       &39304        &4.50536e-04  &2.04  \\\hline
  \end{tabular}  
 \caption{Errors and rates wrt $\| \cdot \|_{h,m}$ for Example~\ref{subsubsec:Curvilinearmovingspatialcomputationaldomain2D}
            using degree $p = 1$ and $p=2.$ }
            \label{table:Curvilinear2DI:discrete}
\end{table}

\begin{table}[htb!]
 \centering
  \begin{tabular}{|l|l|l|l|l|l|l|l|} \hline
    \multicolumn{3}{|c|}{p = 3} &\multicolumn{3}{|c|}{p = 4 } \\
     \cline{1-6}   
       Dofs   &$\|u - u_h \|_{h,m}$   &Rate  &Dofs       &$\|u - u_h \|_{h,m}$ &Rate   \\  \hline
         64   &2.41099e-01    &0      &125   &9.93223e-03   &0 \\
        125   &2.65597e-02    &3.18   &216   &7.60959e-03   &0.38 \\
        343   &4.25800e-03    &2.64   &512   &4.70558e-04   &4.02 \\
       1331   &4.58251e-04    &3.21   &1728  &2.72696e-05   &4.11 \\
       6859   &5.46583e-05    &3.06   &8000  &1.60187e-06   &4.09 \\       
      42875   &6.73180e-06    &3.02   &42875 &9.81655e-08   &4.03 \\\hline
  \end{tabular}  
 \caption{Errors and rates wrt $\| \cdot \|_{h,m}$ for Example~\ref{subsubsec:Curvilinearmovingspatialcomputationaldomain2D}
             using degrees $p = 3$ and $p=4.$}
\label{table:Curvilinear2DII:discrete}
\end{table}


\begin{table}[htb!]
 \centering
  \begin{tabular}{|l|l|l|l|l|l|l|l|} \hline
    \multicolumn{3}{|c|}{p = 1} &\multicolumn{3}{|c|}{p = 2 } \\
     \cline{1-6}   
       Dofs   &$\|u - u_h \|_{L_2(Q)}$  &Rate    &Dofs   &$\|u - u_h \|_{L_2(Q)}$ &Rate   \\  \hline
         8            &3.37961e-01   &0          &27           &2.67611e-02    &0     \\
         27           &7.67535e-02   &2.14       &64           &2.66302e-02    &0.01  \\
         125          &1.95993e-02   &1.96       &216          &2.33938e-03    &3.51  \\
         729          &4.86526e-03   &2.01       &1000         &2.63360e-04    &3.15   \\
        4913          &1.25329e-03   &1.95       &5832         &3.20369e-05    &3.04   \\
       35937          &3.61179e-04   &1.79       &39304        &3.97751e-06    &3.00    \\ \hline
  \end{tabular}  
 \caption{$L_2$ errors and rates for Example~\ref{subsubsec:Curvilinearmovingspatialcomputationaldomain2D}
             using degrees $p = 1$ and $p=2.$}
  \label{table:Curvilinear2DI}
\end{table}
\begin{table}[htb!]
 \centering
  \begin{tabular}{|l|l|l|l|l|l|l|l|} \hline
    \multicolumn{3}{|c|}{p = 3} &\multicolumn{3}{|c|}{p = 4 } \\
     \cline{1-6}   
       Dofs   &$\|u - u_h \|_{L_2(Q)}$   &Rate &Dofs         &$\|u - u_h \|_{L_2(Q)}$ &Rate \\  \hline
         64   &2.65683e-01  &0       &125      &6.83584e-04   &0 \\
        125   &2.26323e-03  &3.55    &216      &6.75251e-04   &0.02 \\
        343   &3.36653e-04  &2.74    &512      &3.96777e-05   &4.09 \\
       1331   &1.70973e-05  &4.30    &1728     &1.60770e-06   &4.63 \\
       6859   &9.93566e-07  &4.10    &8000     &5.11255e-08   &4.97 \\
      42875   &6.12843e-08  &4.02    &46656    &1.61035e-09   &4.99 \\  \hline
  \end{tabular}  
 \caption{$L_2$ errors and rates for Example~\ref{subsubsec:Curvilinearmovingspatialcomputationaldomain2D}
             using degrees $p = 3$ and $p=4.$}
\label{table:Curvilinear2DII} 
\end{table}

\subsection{Parallel Solution for the Case $p=1$}
\label{subsec:ParallelSolution}

We here consider exactly the same problem as in 
Subsection~\ref{subsubsec:Fixedspatialcomputationaldomain2d},
i.e., our parabolic model problem is posed in the 3d space-time domain $Q=(0,1)^3$.
For simplicity, we here consider only the case $p=1$ that turns out 
to be the finite element case 
with trilinear hexahedral finite elements. First, we decompose the space-time mesh into several subdomains and assemble the arising linear systems of Subsection~\ref{subsubsec:Fixedspatialcomputationaldomain2d} in parallel. For example, in Figure \ref{fig:SpaceTimeSubdomains}, the space-time decomposition with $64$ subdomains is shown.
 Afterwards, we solve these linear systems also in parallel with the GMRES method, where we apply the AMG library hypre as a preconditioner. 
For the stopping criteria, we use the relative 
 residual
error reduction by $10^{-10}$. In Table \ref{table:ParallelSolving}, we show the iteration numbers and the solving times for different uniform refinement levels, where we increase the number of cores for 
larger
problems. 
We observe that the iteration numbers are slightly increasing, but the performance of this solver is still very good, even if this solver was not constructed for space-time problems. This example was computed on the supercomputer Vulcan BlueGene/Q in Livermore, California U.S.A by using the 
finite element library MFEM.

\begin{table}[htb!]
 \centering
  \begin{tabular}{|r|c|c|c|c|r|} \hline 
       Dofs   &$\|u - u_h \|_{L_2(Q)}$  &Rate & iter   & time [s] & cores  \\  \hline
	8 & 3.65528e-01 & - & 1 & 0.01 & 1 \\
	27 & 9.39008e-02 & 1.961 & 2 & 0.01 & 1 \\
	125 & 2.32674e-02 & 2.013 & 6 & 0.01 & 1 \\
	729 & 5.75635e-03 & 2.015 & 15 & 0.07 & 64 \\
	4 913 & 1.43198e-03 & 2.007 & 16 & 0.14 & 64 \\
	35 937 & 3.57217e-04 & 2.003 & 19 & 0.40 & 64 \\
	274 625 & 8.92171e-05 & 2.001 & 24 & 1.04 & 1 024 \\
	2 146 689 & 2.22941e-05 & 2.001 & 29 & 3.65 & 1 024 \\
	16 974 593 & 5.57231e-06 & 2.000 & 36 & 21.40 & 1 024 \\
	135 005 697 & 1.39293e-06 & 2.000 & 50 & 36.26 & 8 192 \\
	1 076 890 625 & 3.48206e-07 & 2.000 & 63 & 156.50 & 16 384 \\ \hline
  \end{tabular}  
\caption{Solver performance for Example~\ref{subsubsec:Fixedspatialcomputationaldomain2d}  and $p=1$.}
  \label{table:ParallelSolving}
\end{table}

\begin{figure}[htb!]
     \begin{center}
       \includegraphics[width = 0.7\textwidth]{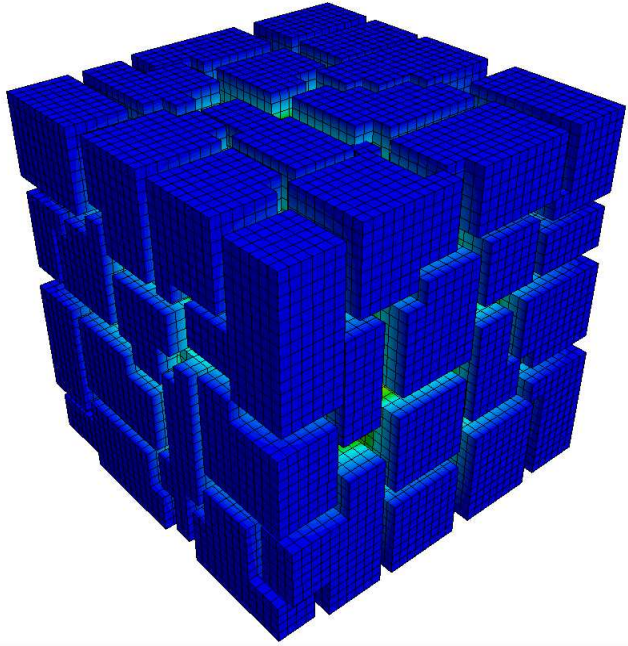}
    \end{center}
    \caption{Space-time decomposition with $64$ subdomains.} 
    \label{fig:SpaceTimeSubdomains}
\end{figure}


\section{Conclusions}
\label{Conclusions}

We have introduced the Space-Time IgA for parabolic evolution problems. 
We have presented a-priori error
estimates and numerical examples in the space-time domain 
$Q = \{(x,t) \in \mathbb{R}^{d+1}: x \in \Omega(t), t \in (0,T)\}$ 
for both fixed spatial domains $\Omega \subset \mathbb{R}^{d}$ 
and moving spatial domains $\Omega(t) \subset \mathbb{R}^{d}$, $t \in [0,T]$. 
Our numerical experiments have been preformed 
on a sequence of refined meshes and for the polynomial degrees $p=1,2,3,4$ 
in the cases $d=1$ as well as $d=2$. 
In the case of smooth solutions, we have nicely observed 
the full asymptotical convergence rates. 
For simplicity, we restricted our-self to the single-patch case.
However, it is 
possible
to generalize the results to the conform 
multi-patch case.
Moreover, the combination of the results of this paper 
with the results of the papers \cite{LMN:LangerToulopoulos:2014a} 
and \cite{LMN:Neumueller:2013a}
allows us to analyze space-time multi-patch dG  IgA schemes
in a similar way. 
The overall efficiency of the space-time IgA heavily depends on 
the availability of fast parallel solvers. 
At the first glance, the solution of one large space-time system of 
linear algebraic equations instead of many smaller systems 
in traditional time-stepping methods seems to be a big disadvantage
of space-time IgA, but on parallel computers with many cores  
this is a big advantage that allows us to overcome 
the curse of sequentiality as Example~\ref{subsec:ParallelSolution} shows. 
Another advantage consists in the elegant treatment of moving 
domains or interfaces. And last, but not least the 
space-time adaptivity or, more precisely, the possibility to 
perform a free adaptivity in $Q$ without separating the time from the
space ($t$ is just another variable $x_{d+1}$) 
opens new horizons in developing highly efficient parallel 
adaptive space-time IgA methods for parabolic as well as 
hyperbolic problems, using T-splines
\cite{LMN:ScottSimpsonEvansLiptonBordasHughesSederberg:2013a}
or THB-splines 
\cite{LMN:GiannelliJuettlerSpeleers:2012a}
for local refinement.

\section*{Acknowledgment}
The research is supported by the Austrian Science Fund (FWF) through the
NFN S117-03 project. 
We also want to thank P. Vassilevski for the possibility to compute on the Vulcan Cluster in Livermore. 
Especially M. Neum\"{u}ller wants to thank him for the support during his one month 
visit at the Lawrence Livermore National Laboratory.

\bibliographystyle{plain} 
\bibliography{igaspacetime}

\end{document}